\documentclass[11pt, oneside, wide, times]{preprintYSH}

\usepackage[english]{babel}
\usepackage{array}
\usepackage{tikz}
\usepackage{mathYSH}

\makeatletter
\renewcommand{\@MHputnumber}{\if@MHhaslab\@MHstep\if@MHleft\@MHwritelabel\fi%
        \hbox{\textnormal{(\@MHeqno)}}%
	\if@MHleft\else\@MHwritelabel\fi%
	\fi\@MHresetlab}
\makeatother

\renewcommand{\crochet}[1]{\left\langle #1 \right\rangle}
\renewcommand{\norm}[1]{\left\| #1 \right\|}
\renewcommand{\normm}[1]{{\left\vert\!\left\vert\!\left\vert #1 \right\vert\!\right\vert\!\right\vert}}

\newcommand{\rh}{\mathrm{h}}
\newcommand{\rZ}{\mathrm{Z}}

\begin{document}

\title{Landau Hamiltonian with Gaussian white noise potential and the asymptotics of the bottom of its spectrum}
\author{Yueh-Sheng Hsu\footnote{\email{yueh-sheng.hsu@tuwien.ac.at}}}
\institute{TU Wien}

\vspace{2mm}

\date{\today}
\maketitle

\begin{abstract}
	We present a simple construction of a random Schrödinger operator subject to a magnetic field with a regularity as low as $0^-$-Hölder and a Gaussian white noise electric potential on a two-dimensional bounded box. This construction is based on the exponential Ansatz introduced in \cite{HL15} and leverages the semigroup approach developed in \cite{HL26}.
	
	The proposed construction enables us to generalise an asymptotic result for the bottom of the spectrum of the two-dimensional continuous Anderson Hamiltonian, first proved in \cite{CvZ21}, to the magnetic case. Our choice of potential not only covers the case of a uniform magnetic field, but also cases that break translation invariance.
	\medskip
	
	\noindent
	{\bf MSC2020 subject classifications}: Primary 35J10, 60H15; Secondary 47A10. 
\end{abstract}

\section{Introduction}\label{sec:landau_intro}
We consider the random Schrödinger operator of the form
\begin{equ}\label{eq:Landau}
	H := (i \nabla + \bA)^2 + V
\end{equ}
defined on a two-dimensional bounded domain. In the context of quantum mechanics, \eqref{eq:Landau} arises from quantising the Hamiltonian (i.e., the total energy) of a single-particle system subjected to the electric field $-\nabla V$ and the magnetic field $\nabla \times \bA$. Here, the scalar field $V$ and the vector field $\bA$ are referred to as the electric and magnetic potentials, respectively, and both of them can be random fields. In this work, we are interested in the case where the electric potential $V$ is modelled by the Gaussian white noise $\xi$, while $\bA$ is given by a vector-valued function of class $1^-$-Hölder.\\

Let us begin by motivating the model. When $V$ vanishes and $\bA$ is chosen such that the magnetic field $\nabla \times \bA$ is uniform with strength $1$, the operator \eqref{eq:Landau} is called the \emph{Landau Hamiltonian}. This operator has the notable property that, as an unbounded operator on $L^2(\R^2)$, its spectrum consists of eigenvalues $2n+1$, where $n = 0, 1, 2, \dots$, each with infinite multiplicity. These eigenvalues are referred to as the \emph{Landau levels}, and they are closely related to the \emph{quantum Hall effect}: when a thin conducting sample is subjected to a perpendicular magnetic field, an induced voltage appears in the lateral direction across the sample; in the quantum regime, the ratio of this voltage to the current strength (i.e., the Hall resistance) is quantised and can be expressed in terms of the Landau levels. A physically relevant consideration is the presence of material impurities, modelled by a random perturbation $V$, in the conducting sample. This leads us to study an operator of the form \eqref{eq:Landau}, and a natural question to pursue is how the random perturbation affects the spectrum of $H$.

On the other hand, \eqref{eq:Landau} can also be viewed as a generalisation of the \emph{continuous Anderson Hamiltonian} $H = -\Delta + V$, i.e., when $\bA = 0$ and $V$ is random. The Anderson Hamiltonian was originally formulated in a discrete setting; its spectral behaviour has sparked great interest and remains a focus of intense research. Although the continuous analogue of the Anderson Hamiltonian can be more technically challenging, even in terms of its definition, it is widely believed that intuition from its discrete counterpart carries over to the continuous setting. In particular, it is conjectured that as long as $V$ is \emph{sufficiently random}, the spectrum of $H$ becomes pure point, and its eigenfunctions are exponentially localised. This phenomenon is referred to as \emph{Anderson localisation}. Numerous efforts and results have been devoted to rigorously establishing this phenomenon under various conditions. A natural question, therefore, is how these results are affected or modified by the introduction of a magnetic field.\\

Among all random potentials, we are particularly interested in the case of \emph{Gaussian white noise}, where the random potential $V$ is a $\delta$-correlated Gaussian field. While this assumption is considered a simplification in the physics literature due to the short-range correlation of the field, the construction of such an operator is highly non-trivial from a mathematical perspective. Indeed, the Gaussian white noise should be interpreted as a random distribution, in the sense that $\{\xi(\varphi); \varphi \in C^\infty_c\}$ is a Gaussian family with covariance
\[\E[\xi(\varphi_1) \xi(\varphi_2)] = \crochet{\varphi_1, \varphi_2}_{L^2}, \quad \varphi_1, \varphi_2 \in C^\infty_c.\]
It is well known that, almost surely, $\xi$ is only an $\alpha$-Hölder distribution with $\alpha < -1$ in dimension $2$. Therefore, the operator \eqref{eq:Landau} does not make sense as a sum of the Laplacian and the potential $\xi$ interpreted as a multiplication operator: since the operator domain of $H$ contains functions at best $(2+\alpha)$-Hölder, the product between the domain functions and $\xi$ is in general undefined as $\alpha + (2+\alpha) < 0$. Such operators belong to the class of \emph{singular Schrödinger operators}, a name which echoes that of \emph{singular stochastic PDEs}: indeed, the breakthroughs in the latter domain, notably the theories of regularity structures \cite{Hai14} and of paracontrolled distributions \cite{GIP15}, have proved essential in their study.

Two separate issues arise in the construction of such operators. The first is to give a meaning to the ill-defined products through \emph{renormalisation}; in the Anderson case $\bA = 0$, to which most of the literature is devoted, this is resolved either within the theory of paracontrolled distributions \cite{AC15, GUZ20, CvZ21, Mou22} or that of regularity structures \cite{Lab19, MvZ25}.
The second issue is to extract a self-adjoint operator from the renormalised data; this has been achieved by a direct description of the operator and its (random) domain \cite{AC15, GUZ20, CvZ21}, or through one of the operator's avatars: its resolvent \cite{Lab19}, its Dirichlet form \cite{MvZ25}, or its semigroup \cite{HL26}.

These constructions cover a variety of settings. The first goes back to Allez and Chouk \cite{AC15}, who studied the operator on the two-dimensional torus $\T^2$. The construction was subsequently extended to bounded domains in higher dimensions and/or to various boundary conditions (e.g., on $\T^3$ \cite{GUZ20, Lab19}, bounded boxes of dimensions $2$ or $3$ with zero Dirichlet boundary conditions \cite{Lab19, CvZ21, MvZ25}, zero Neumann boundary conditions \cite{CvZ21, MvZ25}, and compact manifolds \cite{Mou22}), as well as to the full space $L^2(\R^d)$ (see \cite{Uek23, Ugu22, HL26} for the case $d = 2$ and \cite{HL26} for $d=3$). However, the literature for the case of non-vanishing magnetic fields with white noise perturbations is very limited. To the best of our knowledge, the only work in this direction is by Morin and Mouzard \cite{MM22}, who considered the case where the magnetic field itself is modelled by Gaussian white noise.

Beyond the construction of these operators, the operator spectrum is another key to understanding the localisation phenomenon as well as the so-called \emph{intermittency} of the associated parabolic PDE, $\partial_t u = -H u$. Intermittency refers to the situation where most of the mass of the solution $u$ concentrates on isolated small islands. Investigating intermittency often requires understanding $U^{B_t}$, the total mass of the solution within a box $B_t$ whose size grows with time $t$. Using the spectral expansion of $H$, one can heuristically approximate the large-time asymptotics of $U^{B_t}$ by
\[U^{B_t}(t) \approx e^{-t \lambda_1^{B_t}},\]
where $\lambda_1^{B_t}$ is the lowest eigenvalue of $H$ on the box $B_t$. Consequently, understanding the asymptotics of $\lambda_1^{B_t}$ as $t \to \infty$ becomes crucial.

In the Anderson case with white noise perturbation, progress has been made on this topic, particularly concerning the asymptotics of eigenvalues as the box size increases. For example, Chouk and van Zuijlen \cite{CvZ21} obtained results in dimension $2$, while the work of Labbé and the author \cite{HL23} addressed dimension $3$. Based on these asymptotic results, König, Perkowski and van Zuijlen \cite{KPv22} deduced the large-time asymptotics of the total mass of the solution to the continuous parabolic Anderson model (PAM) in dimension $2$, while Ghosal and Yi \cite{GY23} studied the large-time asymptotics of the PAM solution as well as its fractality in dimensions $2$ and $3$.\\

The purpose of the present work is twofold: we will provide a simple construction of the operator in the singular case with non-vanishing magnetic field, and then study the asymptotics of its lowest eigenvalues.

As our first main result, we will construct the operator \eqref{eq:Landau} on any bounded open box with Gaussian white noise potential $V=\xi$ and with any magnetic vector potential $\bA$ in the class $\cC_{\mathrm{loc}}^{1-\kappa}(\R^2;\R^2)$ of locally $(1-\kappa)$-Hölder functions, for all $\kappa \in (0, 
\frac12)$. The construction proposed here does not rely on advanced singular SPDE theories such as regularity structures or paracontrolled distributions, but instead uses the exponential transformation trick first introduced in \cite{HL15}. Additionally, we utilise the Hille--Yosida theory to extract the desired operator from the solution theory of parabolic SPDEs. Let us mention that each of the two ingredients of our construction has appeared separately in the literature: the exponential transformation was employed in the Dirichlet-form construction of \cite{MvZ25}, and the extraction of the operator from its semigroup in \cite{HL26}; their combination, however, appears to be new, and it is this combination that keeps the whole construction elementary. As in the previously cited works, the construction is carried out in the sense of \emph{renormalisation}. Specifically, we fix a compactly supported, radially symmetric function $\rho \in C^\infty_c(\R^2)$, define the mollifier sequence $\rho_\eps = \eps^{-2}\rho(\cdot/\eps)$, and consider the regularised noise $\xi^{(\eps)} = \rho_\eps \ast \xi$. With the renormalisation constant $C^{(\eps)}$ defined in Section \ref{sec:Q}, we take the regularised operator
\begin{equ}
	H_\eps := (i\nabla + \bA)^2 + \xi^{(\eps)} + C^{(\eps)}\;,
\end{equ}
which is well-defined and self-adjoint on any bounded open box $B$ with zero Dirichlet boundary condition. We are now ready to state our first main result.
\begin{theorem}\label{thm1}
	For $\kappa \in (0, 1/2)$, let $\bA \in \cC^{1-\kappa}_{\mathrm{loc}}(\R^2; \R^2)$. Fix a bounded open box $B \subset \R^2$. Then, there exists a random operator $H$ on $L^2(B; \C)$ which, almost surely, satisfies the zero Dirichlet boundary condition and is self-adjoint, and is such that
	\[ H_\eps \xrightarrow[\eps \to 0]{} H, \quad \text{in the norm-resolvent sense in $L^2(B; \C)$ and in probability.}\]
	This operator $H$ is therefore regarded as a realisation of the formal expression $(i\nabla + \bA)^2 + \xi$.
	
	Furthermore, almost surely, the random operator $H$ has compact resolvent, and hence its spectrum is pure point and consists of eigenvalues
	\[\lambda^B_1 \leq \lambda^B_2 \leq \lambda^B_3 \leq \cdots.\]
\end{theorem}
\begin{remark}\label{rmk:spectral_gap}
	The eigenvalues in Theorem \ref{thm1} are enumerated in non-decreasing order with multiplicity, and no strict inequality is claimed. In the Anderson case $\bA = 0$, the ground state is known to be simple on a two-dimensional closed manifold: the heat kernel of the associated semigroup is positive and the Krein--Rutman theorem applies; quantitative lower bounds on the gap $\lambda_2 - \lambda_1$ are also available; see \cite[Prop.~4.1, Cor.~4.2 and Prop.~4.16]{BDM25}. When $\bA \neq 0$, the semigroup acts on complex-valued functions and this positivity mechanism seems to break down.
\end{remark}

The second main result of the article concerns the asymptotics, as the size of the box $B$ grows to infinity, of the eigenvalues $\lambda_n^B$. Here, we will specialise to vector potentials $\bA$ verifying Assumption \ref{ass:A} below. For $x \in \R^2$ and $L > 0$, we set $B_{x, L} = x + (-L/2, L/2)^2$, the open box with side length $L$ centred at $x$.
\begin{assumption}\label{ass:A}
	For $\kappa \in (0, \frac12)$, let $\alpha \ge 1- \kappa$ and $0 \leq \gamma < \frac{1+\alpha}{2}$. Assume that $\bA$ is a locally bounded random vector field on $\R^2$ independent of $\xi$ such that almost surely there exist constants $C, L_0 > 0$ for which
	\begin{equ}[eq:A_bound]
		|\bA(x) - \bA(y)| \leq C \log(1 + L)^\gamma |x - y|^\alpha, \quad x, y \in B_{0, L},
	\end{equ}
	for all $L \ge L_0$. In particular, the constants $C$ and $L_0$ can be random.
\end{assumption}
In particular, a random vector field $\bA$ satisfying Assumption \ref{ass:A} belongs to the class $\cC^{1-\kappa}_{\mathrm{loc}}$ almost surely. Setting $\lambda^{L}_n = \lambda^{B_{0, L}}_n$ to be the eigenvalues of $(i\nabla + \bA)^2 + \xi$ on $L^2(B_{0, L})$ given by Theorem \ref{thm1}, we claim the following:
\begin{theorem}\label{thm2}
	Under Assumption \ref{ass:A}, for every $n \geq 1$, one has almost surely that
	\[\lambda_{n}^L \sim - C_{\mathrm{GN}} \log L \quad \text{as $L \to +\infty$,}\]
	where $C_{\mathrm{GN}}$ is the optimal constant for the Gagliardo--Nirenberg inequality in dimension $2$, i.e.,
	\begin{equ}\label{eq:Gagliardo--Nirenberg}
		C_{\mathrm{GN}} = \sup_{f \in \cH^1 \setminus \{0\}} \frac{\norm{f}_{L^4}^4}{\norm{\nabla f}_{L^2}^{2} \norm{f}_{L^2}^{2}}.
	\end{equ}
\end{theorem}
\noindent Theorem \ref{thm2} is therefore an extension of \cite{CvZ21} to the magnetic case. A major difference in our case is that our operator, and therefore its eigenvalues, are no longer translation invariant, as opposed to the Anderson Hamiltonian. This difficulty is handled by Proposition \ref{prop:Landau-Anderson} below: after gauging away the value $\bA(x)$ at a fixed reference point $x$ from the field $\bA$ within a box, the principal eigenvalue can be compared with that of the Anderson Hamiltonian, with an error controlled by the oscillation of $\bA$ in that box, which Assumption \ref{ass:A} keeps small at the relevant scales.\\

Notably, the following choices of magnetic potential are covered by the present scope.
\begin{example}[Uniform magnetic field]
	One can consider the case where the magnetic field $\nabla \times \bA$ is independent of the space variable. A choice of vector potential for such a case is given by
	\begin{equ}\label{eq:A_unif}
		\bA(x_1, x_2) = \frac{1}{2} (-x_2, x_1), \quad (x_1, x_2) \in \R^2.
	\end{equ}
	In particular, such $\bA$ satisfies Assumption \ref{ass:A} with $\alpha = 1$ and $\gamma = 0$. 
\end{example}
\begin{example}[Gaussian free field]
	One can also consider rough magnetic fields. For example, let $\Gamma$ be the Gaussian free field (independent of $\xi$) and consider vector potentials associated with the magnetic field given by $\Gamma$. We can take $\bA = \nabla^{\perp} \Delta^{-1} \Gamma$, where $\nabla^{\perp} = (-\partial_2, \partial_1)$. In particular, one has $\nabla \times \bA = \Gamma$, and $\bA \in \cC^{1-\kappa}_{\mathrm{loc}}(\R^2;\R^2)$ almost surely for any $\kappa \in (0,\frac12)$. Moreover, it is well known that this $\bA$ satisfies Assumption \ref{ass:A} with $\alpha = 1-\kappa$ and $\gamma = \frac12$.
\end{example}

Let us now outline the ideas of the proofs of Theorems \ref{thm1} and \ref{thm2}. The key idea is to reformulate the problem of operator construction into the construction of a semigroup. More precisely, we consider the singular SPDE:
\[\partial_t u = -(i\nabla + \bA)^2 u - \xi u, \quad u(t = 0, \cdot) = g \in L^2.\]
Notice that when $\bA = 0$, this PDE reduces to the parabolic Anderson model (PAM), one of the archetypal examples in the breakthroughs of singular SPDE theories \cite{Hai14, GIP15}. We argue that our parabolic equation can be solved using arguments similar to those developed for PAM in \cite{HL15}. Furthermore, we deduce that the solution map induces a semigroup $(P_t)$ of bounded operators which is strongly continuous in the time variable. This allows us to apply a classical result due to Hille and Yosida (see, for example, \cite[Theorem II.3.8]{EN00}), yielding the infinitesimal generator of the semigroup, which turns out to be our desired operator. This line of reasoning is in the same spirit as a recent work by Labbé and the author \cite{HL26}, which addressed the construction of the Anderson Hamiltonian on the full spaces $\R^2$ and $\R^3$. While their case is much more technical since the semigroup consists of unbounded operators to which the Hille--Yosida theory does not apply, our present setting allows us to carry out the same spirit of proof without too much of a technical burden. As a consequence of the scheme, the construction of $H$ can be factorised into a composition of several maps, as illustrated in Figure \ref{fig:scheme}.

\begin{figure}[h]
	\centering
	\setlength{\tabcolsep}{3pt}
	\begin{tabular}{ccccccc}
		$\Omega$ & $\xrightarrow{\ \text{renormalisation}\ }$ & $\ccA \times \ccM$ & $\xrightarrow{\ \text{deterministic PDE}\ }$ & $\cL(L^2(B; \C))$& $\xrightarrow{\ \text{spectral theory}\ }$ & $\R$\\
		$\xi$ & $\mapsto$ &$(\bA, Q(\xi))$ &$\mapsto$& $P_t(\bA, Q(\xi))$ &$\mapsto$& $\lambda_n(\bA, Q(\xi))$
	\end{tabular}
	\caption{Scheme for the construction of $H$ and its eigenvalues $\lambda_n, n \ge 1$. Here, $\Omega$ denotes the canonical sample space for $\xi$, $\ccA$ is a function space accommodating the field $\bA$ (in our case, $\ccA$ should be thought of as $\cC^{1-\kappa}_{\mathrm{loc}}(\R^2;\R^2)$), and $\ccM$ is an auxiliary space designed to accommodate the \emph{models} $Q$ (in our case, $\ccM$ should be thought of as $\cC^{1-\kappa}_{\mathrm{loc}}(\R^2) \times \cC^{-\kappa}_{\mathrm{loc}}(\R^2))$. The first arrow is the renormalisation step, the second the resolution of a deterministic PDE, and the third an argument of spectral theory.}
	\label{fig:scheme}
\end{figure}

Notice that we have introduced an auxiliary space $\ccM$ in the scheme. This space plays a role analogous to the \emph{models} or \emph{enhanced noises} introduced in the theories of regularity structures \cite{Hai14} and paracontrolled distributions \cite{GIP15}. Concretely, $\ccM$ is the space on which the renormalisation procedure is performed: an element $q = (h, \zeta) \in \ccM$ is a pair of deterministic fields serving as placeholders for the renormalised stochastic objects that we will construct in Section \ref{sec:Q}. The random variable $Q(\xi)$ is nothing but this pair of renormalised fields. While the construction of the operator $P_t(\bA, q)$ for any given instance $(\bA, q) \in \ccA \times \ccM$ is purely analytical, the map $Q: \Omega \to \ccM$ has to be defined via probabilistic arguments. Let us also stress that the vector potential $\bA$ enters the scheme as a fixed element of $\ccA$. When $\bA$ is random, as allowed by Assumption \ref{ass:A} and illustrated by the Gaussian free field example, it is assumed independent of $\xi$; all the constructions below are then performed conditionally on $\bA$, that is, with $\bA$ fixed to a typical realisation of its law (see also the beginning of Section \ref{sec:asymptotics}). An advantage of this scheme is that, except for the map $Q$, every other mapping in Figure \ref{fig:scheme} is continuous. In particular, this implies that the eigenvalues $\lambda_n$ are continuous functions of the data $(\bA, q) \in \ccA \times \ccM$.

This brings us to the proof of Theorem \ref{thm2}. Our argument follows the outline set out in \cite{CvZ21}: as noted there, the asymptotic result claimed in Theorem \ref{thm2} is essentially a consequence of a probability tail bound for the eigenvalues (Theorem \ref{thm:ev_tail_estimate}), which in turn follows from a large deviation principle for the random variable $\xi \mapsto Q(\xi)$ along with certain properties of the function $(\bA, q) \mapsto \lambda_n(\bA, q)$, particularly its continuity. An additional difficulty introduced by the presence of $\bA$ is that the eigenvalues lose their invariance in law under translation, as opposed to the Anderson case. We overcome this problem by using Proposition \ref{prop:Landau-Anderson}, which compares the magnetic eigenvalues to the Anderson ones.\\

As the present work can be seen as a continuation of the work \cite{HL23}, which provides a proof of Theorem \ref{thm2} in the Anderson case $\bA = 0$, let us comment on some minor technical improvements. The work \cite{HL23} builds on an earlier result by Labbé \cite{Lab19}, who constructed the operator $H = -\Delta + \xi$ through its resolvent $S_a = (H + a)^{-1}$ for sufficiently large parameters $a > 0$. In that approach, the operator $S_a(q)$ is constructed for $a \ge a_0(q)$, where the threshold $a_0(q)$ depends on the model $q$ (the family $(S_a(q))_a$ can subsequently be extended to all $a > -\lambda_1(q)$ by the resolvent identity). However, since no single parameter $a$ is admissible for all models $q \in \ccM$ simultaneously, for any fixed $a$ the map $q \mapsto S_a(q)$ is only known to be continuous on the subspace $\ccM_a = \{q \in \ccM: a_0(q) \leq a\}$ of models. As a consequence, probabilistic properties of the eigenvalues could only be proved conditionally on the event $E_a = \{Q \in \ccM_a\}$, which is not of full probability; this explains in part why \cite[Proposition 2.6]{HL23} only obtained a partial large deviation estimate for sets of the form $(-\infty, m)$ with $m \in \R$, while a full large deviation principle should also be true. In contrast, the semigroup approach adopted in the present work involves no such threshold: the operator $P_t(q)$ is constructed for every instance $q \in \ccM$, and the map $q \mapsto P_t(q)$ is continuous on all of $\ccM$. Moreover, the operator $P_t(q)$ is positive by construction, so that for each fixed $n$, its $n$-th eigenvalue $\mu_n(q) = e^{-t\lambda_n(q)}$ is strictly positive and, by continuity in $q$, locally bounded away from $0$; the composition with the function $\mu \mapsto -\frac{1}{t}\log(\mu)$, which is continuous on $(0, \infty)$, therefore poses no problem, and each eigenvalue $\lambda_n$ depends continuously on $q \in \ccM$. Consequently, we are able to establish a full large deviation principle for $\lambda_1$ (Proposition \ref{prop:LDP}).\\

One may also wonder why we do not construct $H$ through its Dirichlet form: quadratic forms appear naturally in Section \ref{sec:asymptotics}, and on a bounded box, form constructions of Anderson-type Hamiltonians are indeed available \cite{GUZ20, MvZ25}. Besides the continuity and positivity properties of $q \mapsto P_t(q)$ discussed above, the semigroup construction has the advantage of providing an explicit description of the operator domain. A construction via forms typically proceeds through the Friedrichs extension: one proves a lower bound for the form on a dense class $\ccD$ of test functions and obtains a self-adjoint operator by the representation theorem; the operator domain produced in this way is however characterised only abstractly, and the class $\ccD$ is in general neither a core of the resulting operator nor explicitly related to its form domain (actually, it can happen that $\ccD$ intersects the operator domain only at $\{0\}$!). By contrast, the Hille--Yosida theorem identifies the operator domain explicitly, as the set of functions at which the semigroup is differentiable at $t = 0$.\\

The remainder of the article is organised as follows. In Section \ref{sec:construction}, we discuss the construction of the random operator $H$ and prove Theorem \ref{thm1}. Specifically, we construct the deterministic map $(\bA, q) \mapsto P_t(\bA, q)$ in Section \ref{sec:P}, and the random data $\xi \mapsto Q(\xi)$ in Section \ref{sec:Q}. In Section \ref{sec:asymptotics}, we explore the asymptotics of the eigenvalues of $H$ and prove Theorem \ref{thm2}. In particular, we reduce the assertion to a probabilistic tail bound (Theorem \ref{thm:ev_tail_estimate}), prove the key lemmas in Section \ref{sec:ingredients}, and establish the tail bound in Section \ref{sec:tail_bound}. Finally, Section \ref{sec:outlook} discusses several generalisations of our results and related open questions.

\subsubsection*{Notations}
\begin{itemize}
	\item We will work with bounded open boxes, i.e., sets of the form $B = \prod_{j=1}^{2} (a_j, b_j)$ with $-\infty < a_j < b_j < \infty$. For $x \in \R^2$ and $L > 0$, we will denote by $B_{x, L}$ the open box $\{y \in \R^2 : |y - x|_\infty < L/2\}$.
	\item For $\alpha \in \R \setminus \N$, we denote by $\cC^\alpha_{\mathrm{loc}}(\R^2)$ the family of locally $\alpha$-Hölder functions or distributions on $\R^2$, as defined in \cite{Hai14}. When a bounded open box $B$ is specified, $\cC^\alpha(B)$ denotes the family of $\alpha$-Hölder functions or distributions on $\bar{B}$. When the underlying domain is clear from the context, we may simply write $\cC^\alpha$.
	\item For $\alpha \in \R_+$ and a bounded open box $B$, $\cH^\alpha(B)$ denotes the Sobolev space on $\bar{B}$. When the underlying domain is clear from the context, we may simply write $\cH^\alpha$.
	\item In the sequel, $\kappa$ denotes a fixed constant in the interval $(0, \frac12)$.
	\item Let $(\Omega, \cF, \P)$ with $\Omega = \cC^{-1-\kappa}_{\mathrm{loc}}$ be the canonical probability space for Gaussian white noise $\xi$.
	\item Given a bounded open box $B$, we define the Banach space $\ccM(B)$ as the
	closure of smooth functions in $\cC^{1-\kappa}(B)\times\cC^{-\kappa}(B)$,
	equipped with the usual norm, and $\ccA(B)$ as the closure of smooth functions
	in $\cC^{1-\kappa}(B;\R^2)$.
	\item Define $\ccM$ as the closure of $C^\infty(\R^2)\times C^\infty(\R^2)$ in
	$\cC^{1-\kappa}_{\mathrm{loc}}(\R^2)\times\cC^{-\kappa}_{\mathrm{loc}}(\R^2)$
	with respect to the metric
	\[d_\ccM(q_1,q_2)=\sum_{n\ge1}2^{-n}\left(1\wedge\norm{q_1-q_2}_{\ccM(B_{0,n})}\right),
	\quad q_1,q_2\in\ccM,\]
	and define $\ccA$ similarly. Then $\ccM$ and $\ccA$ are Polish spaces, and
	smooth elements are dense in them.
	Note that for $\kappa' \in (0, \kappa)$, $\cC^{1-\kappa'}_{\mathrm{loc}}(\R^2;\R^2)$ embeds continuously into $\ccA$; a similar embedding holds for $\ccM$.
\end{itemize}

\paragraph{Acknowledgements.} The author is funded by the European Union (ERC, SPDE, 101117125). Views and opinions expressed are however those of the author(s) only and do not necessarily reflect those of the European Union or the European Research Council Executive Agency. Neither the European Union nor the granting authority can be held responsible for them.

The author would like to thank Cyril Labbé for fruitful discussions on the topic.

\section{Construction of $H$ on a bounded open box} \label{sec:construction}
The aim of this section is to prove Theorem \ref{thm1}. In particular, we will make rigorous the scheme presented in Figure \ref{fig:scheme} by constructing the map $P_t$ (Section \ref{sec:P}) and the random variable $Q$ (Section \ref{sec:Q}). In this section, we fix a bounded open box $B$.

Let us motivate our argument. Remember that we would like to solve the initial value problem $\partial_t u = -(i\nabla + \bA)^2u - \xi u, \quad u(t = 0, \cdot) = g \in L^2$. The following observation due to \cite{HL15} is the key: if we take the Green function $G$ of the operator $-\Delta$ and make the Ansatz $u = e^{-G\ast\xi} v$, then the function $v$ formally solves
\[\partial_t v =  -(i\nabla + \bA)^2 v - 2\nabla (G\ast\xi) \cdot \nabla v  + 2i\nabla (G\ast\xi) \cdot \bA v  + |\nabla (G\ast\xi)|^2 v.\]
Notice that the white noise $\xi$ disappears from the PDE, and the only singular term remaining is $|\nabla (G\ast\xi)|^2$, which can be dealt with via Wick renormalisation (Lemma \ref{lem:conv_noise}).

Consequently, the auxiliary space $\ccM$ is devised so as to accommodate the typical realisations of the random fields $G\ast\xi$ and the Wick renormalised $|\nabla (G\ast\xi)|^2$. Then, given each instance $q \in \ccM$, we can solve the PDE deterministically.

\subsection{Construction of $P_t^B$}\label{sec:P}
Fix a bounded open box $B$. Given $q = (h, \zeta) \in \ccM$, $\bA \in \ccA$ and $g \in L^2(B; \C)$, we consider the deterministic PDE
\begin{equ}\label{eq:Landau_PDE}
	\begin{cases}
		[\partial_t + (i\nabla + \bA)^2 + 2\nabla h \cdot \nabla  - 2i\nabla h \cdot \bA  + \zeta] v(t, x) = 0, & t > 0, x\in B,\\
		v(0, x) = g(x) e^{h(x)}, & t = 0, x \in B,\\
		v(t, x) = 0, & t\ge 0, x \in B^c.
	\end{cases}
\end{equ}
In view of the discussion above, the pair $q = (h, \zeta)$ should be thought of as a deterministic surrogate for the renormalised fields $(G \ast \xi, \, -:|\nabla (G\ast\xi)|^2: - F \ast \xi)$ constructed in Lemma \ref{lem:conv_noise}; in the present subsection, however, no probability is involved and $q$ is a fixed element of $\ccM$.

Throughout this section, we fix the exponents $\theta, \kappa' > 0$ such that
\begin{equ}\label{eq:theta}
	1 + \kappa < \theta < 2 - \kappa - \kappa'\;.
\end{equ}
Note that such exponents exist: since $\kappa < \frac12$, one can first choose $\kappa' \in (0, 1 - 2\kappa)$ and then $\theta$ in the non-empty interval $(1 + \kappa, 2 - \kappa - \kappa')$.
The space $\cH^\theta(B)$ will serve as the target space for the solutions of the PDE \eqref{eq:Landau_PDE}.

Let us define the functions $F_1$, $F_2$, taking values in $\cC^{-\kappa}_{\mathrm{loc}}(\R^2;\C^2)$, $\cC^{-\kappa}_{\mathrm{loc}}(\R^2; \C)$, respectively, by
\begin{equs}
	F_1(\bA, h) &:= 2i\bA + 2\nabla h,\\
	F_2(\bA, h, \zeta) &:= i\nabla \cdot \bA - 2i\nabla h \cdot \bA + |\bA|^2 + \zeta.
\end{equs}
Note that the product $\nabla h \cdot \bA$ appearing above is well-defined in $\cC^{-\kappa}$ and is continuous with respect to its two components by Young's condition $(1-\kappa) + (-\kappa) > 0$ (see \cite{BCD11}), which holds precisely because $\kappa < \frac12$.
The following quantities will appear repeatedly in the sequel:
\begin{equs}[eq:C-and-M]
	C^B(\bA, q) := \norm{F_1}_{\cC^{-\kappa}(B)} + \norm{F_2}_{\cC^{-\kappa}(B)}\;, \quad M^B_h := 1 \vee \left(\norm{e^{-h}}_{\cC^{1-\kappa}(B)} \norm{e^h}_{\cC^{1-\kappa}(B)}\right)\;.
\end{equs}
In particular, it is not hard to see that $C^B(\bA, q) \leq (1 + \norm{\bA}_{\ccA(B)} + \norm{q}_{\ccM(B)})^2$.
Define also the quantity $t_0 = t_0^B(\bA, q)$ by
\begin{equ}[eq:t0]
	C^B(\bA, q) t_0^{(2-\theta - \kappa - \kappa')/2} = \frac12\;.
\end{equ}

The following Proposition \ref{prop:semigroup} solves the initial value problem \eqref{eq:Landau_PDE} in a sense that is strong enough to imply the existence of a strongly continuous semigroup. We say that a space-time function $v$ is a mild solution to \eqref{eq:Landau_PDE} if it satisfies the fixed point problem
\begin{equ}\label{eq:Landau_mild}
	v(t) = K^B_t \ast (ge^h) - \int_0^t K^B_{t-s}\ast\left(F_1(\bA, h)\cdot \nabla v(s) + F_2(\bA, h, \zeta) v(s)\right) \dd s,
\end{equ}
where $K^B_t$ denotes the Green function of $\partial_t - \Delta$ on $B$ with zero Dirichlet boundary condition. Here, we use the shorthand $v(t) = v(t, \cdot)$ and see it as a function of the time variable taking values in $L^2(\R^2)$.
The existence and uniqueness of the solution to \eqref{eq:Landau_mild} will be understood in the function space $\cE$ given as follows.
For each $t_0 > 0$, define the Banach space
\begin{equ}
	\cE_{t_0} = \left\{v: (0, t_0] \to L^2(\R^2)\;: \normm{v}_{t_0} := \sup_{t \in (0, t_0]} t^{\theta/2} \norm{v(t)}_{\cH^{\theta}(B)} < \infty \right\},
\end{equ}
with $\theta$ verifying \eqref{eq:theta}.
Now let
\begin{equ}
	\cE = \left\{v: [0, \infty) \to L^2(\R^2)\;: v\restriction_{(0, t_0]}\; \in \cE_{t_0} \text{ for all $t_0 > 0$} \right\}.
\end{equ}
\begin{proposition}\label{prop:semigroup}
	For $q = (h, \zeta) \in \ccM$, $\bA \in \ccA$ and $g \in L^2(B; \C)$, the initial value problem \eqref{eq:Landau_PDE} admits a unique, global-in-time mild solution $v$ in $\cE$.
	Define, for $t \in (0, \infty)$,
	\begin{equ}\label{eq:semigroup}
		P^B_t(\bA, q): g \mapsto e^{-h} v(t, \cdot),
	\end{equ}
	and $P_0^B(\bA, q) = \mathrm{Id}$ for $t = 0$. Then, the following assertions hold for the map $(t, \bA, q, g) \mapsto P^B_t(\bA, q) g$ defined on $[0, \infty) \times \ccA \times \ccM \times L^2(B)$.
	\begin{enumerate}
		\item For fixed $(\bA, q) \in \ccA \times \ccM$ and $t \ge 0$, the map $g \mapsto P^B_t(\bA, q) g$ defines a bounded, symmetric (hence self-adjoint), compact and positive operator on $L^2(B)$.
		\item For fixed $t \ge 0$, the map $(\bA, q) \mapsto P^B_t(\bA, q)$ is locally Lipschitz-continuous from $\ccA \times \ccM$ to $\cL(L^2(B))$.
		\item For fixed $(\bA, q) \in \ccA \times \ccM$, the map $t \mapsto P^B_t(\bA, q)$ is continuous from $[0, \infty)$ to $\cL(L^2(B))$.
		\item For fixed $(\bA, q) \in \ccA \times \ccM$ and $0 \leq s < t < \infty$, one has the semigroup property $P^B_{t}(\bA, q) = P^B_{t-s}(\bA, q) \circ P^B_{s}(\bA, q)$.
	\end{enumerate}
\end{proposition}
\begin{proof}
	We first argue that a unique solution to \eqref{eq:Landau_mild} exists up to some time horizon $t_0 = t_0(\bA, q)$ depending only on the data $(\bA, q)$, such that all the assertions hold on the time interval $t \in [0, t_0]$.
	Let $\cM(v)$ be given by the right-hand side of \eqref{eq:Landau_mild}.
	Let us estimate the two terms of $\cM(v)$ separately.
	
	For the first term, the heat-kernel smoothing estimate $\norm{K^B_t \ast f}_{\cH^{\theta}} \lesssim t^{-\theta/2} \norm{f}_{L^2}$ yields
	\begin{equ}
		\sup_{t \le t_0} t^{\theta/2}\norm{K^B_t\ast(ge^h)}_{\cH^\theta} \lesssim \norm{g e^h}_{L^2(B)}\;.
	\end{equ}
	For the second term, since $\nabla v(s) \in \cH^{\theta - 1}(B)$ and $F_1, F_2 \in \cC^{-\kappa}$ with $\theta - 1 + (-\kappa) > 0$, one can utilise the well-known continuity of the product map between Sobolev and Hölder spaces \cite{BCD11}
	\begin{equ}
		\cH^{\theta - 1} \times \cC^{-\kappa} \to \cH^{-\kappa - \kappa'}\;,
	\end{equ}
	for all $\kappa' > 0$, giving
	\begin{equ}\label{eq:integrand_bound}
		\norm{F_1 \cdot \nabla v(s) + F_2 v(s)}_{\cH^{-\kappa - \kappa'}} \lesssim C^B(\bA, q) \norm{v(s)}_{\cH^{\theta}(B)} \leq C^B(\bA, q)\, s^{-\theta/2} \normm{v}_{t_0}\;.
	\end{equ}
	Therefore, the smoothing estimates associated with $K^B_t$ imply
	\begin{equ}
		\norm{\int_0^t K^B_{t-s}\ast\left(F_1 \cdot \nabla v(s) + F_2 v(s)\right) \dd s}_{\cH^\theta} \lesssim C^B(\bA, q) \normm{v}_{t_0} \int_0^t (t-s)^{-\frac{\theta+\kappa+\kappa'}{2}} s^{-\frac{\theta}{2}} \dd s\;.
	\end{equ}
	Thanks to our choice of $\theta < 2-\kappa-\kappa'$, the last time integral is convergent and is of order $t^{1 - \theta - \frac{\kappa+\kappa'}{2}}$.
	
	Altogether, for any $t_0 > 0$ and $v_1, v_2 \in \cE_{t_0}$,
	\begin{align}
		&\normm{\cM(v_1)}_{t_0} \lesssim \norm{g e^h}_{L^2(B)} + C^B(\bA, q)\, t_0^{\frac{2-\theta-\kappa-\kappa'}{2}} \normm{v_1}_{t_0}, \label{eq:M-1}\\
		&\normm{\cM(v_1) - \cM(v_2)}_{t_0} \lesssim C^B(\bA, q)\, t_0^{\frac{2-\theta-\kappa-\kappa'}{2}} \normm{v_1 - v_2}_{t_0}. \label{eq:M-2}
	\end{align}
	Therefore, the affine map $\cM$ is well-defined on $\cE_{t_0}$ and, by choosing $t_0$ sufficiently small such that $C^B(\bA, q)\, t_0^{(2-\theta-\kappa-\kappa')/2} < 1$, it defines a contraction and thus admits a unique fixed point $v \in \cE_{t_0}$. Moreover, given two data tuples $(g, \bA, h, \zeta)$, $(\bar g, \bar\bA, \bar h, \bar\zeta)$, one has the estimate for their corresponding fixed points $v, \bar v$,
	\begin{equs}
		\normm{v - \bar v}_{t_0} &\lesssim \norm{(g - \bar g) e^h}_{L^2(B)} + t_0^{\frac{2-\theta-\kappa-\kappa'}{2}} \Big[\normm{v}_{t_0} \big(\norm{F_1 - \bar F_1}_{\cC^{-\kappa}(B)} + \norm{F_2 - \bar F_2}_{\cC^{-\kappa}(B)}\big)\\
		&\qquad + \normm{v - \bar v}_{t_0} \big(\norm{\bar F_1}_{\cC^{-\kappa}(B)} + \norm{\bar F_2}_{\cC^{-\kappa}(B)}\big)\Big],
	\end{equs}
	where $F_j = F_j(\bA, h, \zeta)$ and $\bar F_j = F_j(\bar\bA, \bar h, \bar\zeta)$ for $j = 1, 2$. From the bounds
	\begin{equs}
		\norm{F_1 - \bar F_1}_{\cC^{-\kappa}} &\lesssim \norm{\bA - \bar \bA}_{\cC^{1-\kappa}} + \norm{h - \bar h}_{\cC^{1-\kappa}},\\
		\norm{F_2 - \bar F_2}_{\cC^{-\kappa}} &\lesssim (1 + \norm{h}_{\cC^{1-\kappa}} + \norm{\bA + \bar \bA}_{\cC^{1-\kappa}})\norm{\bA - \bar \bA}_{\cC^{1-\kappa}} \\&+ \norm{\bA}_{\cC^{1-\kappa}} \norm{h - \bar h}_{\cC^{1-\kappa}} + \norm{\zeta - \bar \zeta}_{\cC^{-\kappa}},
	\end{equs}
	it follows that, by shrinking $t_0$ so that $C' t_0^{(2-\theta-\kappa-\kappa')/2} < 1$ for some $C' = C'(\bA, h, \zeta, \bar\bA, \bar h, \bar\zeta)$,
	\begin{equ}\label{eq:continuity_fixed-pt}
		\begin{aligned}
		\normm{v - \bar v}_{t_0} &\lesssim \frac{C'}{1 - C' t_0^{\frac{2-\theta-\kappa-\kappa'}{2}}}\Big[\norm{g - \bar g} \\
		&\qquad + t_0^{\frac{2-\theta-\kappa-\kappa'}{2}} \left(\norm{\bA - \bar \bA}_{\cC^{1-\kappa}} + \norm{h - \bar h}_{\cC^{1-\kappa}} + \norm{\zeta - \bar \zeta}_{\cC^{-\kappa}}\right)\Big].
		\end{aligned}
	\end{equ}
	
	One thus defines $P^B_t(\bA, h, \zeta) g = e^{-h} v(t)$. Then for all $0 < t \leq t_0$, it is clear from \eqref{eq:M-1} that $g \mapsto P^B_t(\bA, q) g$ is linear and, since multiplication by $e^{-h} \in \cC^{1-\kappa}$ maps $\cH^{\theta}(B)$ into $\cH^{1-\kappa}(B)$ (recall $\theta > 1 + \kappa > 1 - \kappa$, and $\cH^{\theta} \hookrightarrow L^\infty$ as $\theta > 1$ by Sobolev embedding \cite{BCD11}),
	\begin{equ}\label{eq:P_bound_sobolev}
		\norm{P^B_t(\bA, q)}_{\cL(L^2;\cH^{1-\kappa})} \lesssim \frac{\norm{e^{-h}}_{\cC^{1-\kappa}(B)}\norm{e^h}_{\cC^{1-\kappa}(B)}}{1- C^B(\bA, q)\, t^{\frac{2-\theta-\kappa-\kappa'}{2}}}\; t^{-\theta/2}\;.
	\end{equ}
	A fortiori, the operator $P^B_t(\bA, q)$ is bounded on $L^2(B)$. Items 2--4 in the statement now follow easily: since the constant $C'$ is finite provided that the data $\bA, h, \zeta, \bar\bA, \bar h, \bar\zeta$ are bounded, \eqref{eq:continuity_fixed-pt} implies the local Lipschitz continuity claimed in item 2. The semigroup property (item 4.) is a direct consequence of the mild form \eqref{eq:Landau_mild}. For item 3., it suffices to prove the strong right-continuity of $t \mapsto P^B_t$ at $t = 0$ (since then the full continuity follows from the semigroup property and the boundedness of $P^B_t$). To see this, we write (recall that $v(s) = e^h P_s^B g \in \cH^{\theta}$ and $\int_0^\eps \norm{v(s)}_{\cH^{\theta}} \dd s \lesssim \eps^{1 - \theta/2} \normm{v}_{t_0}$, with $1 - \theta/2 > 0$)
	\[(P_\eps^B - I)g = e^{-h} (K_\eps^B - I)(e^h g) + e^{-h} \int_0^\eps K_{\eps-s}^B \ast [F_1 \cdot \nabla v(s) + F_2 v(s)] \dd s,\]
	where both terms converge in $L^2(B)$ to $0$ as $\eps \to 0$.
	
	It now remains to demonstrate item 1. Since $P^B_t(\bA, h, \zeta)$ takes values in $\cH^{1-\kappa}(B)$, which can be compactly embedded into $L^2(B)$ by the Rellich--Kondrachov theorem, it is a compact operator on $L^2(B)$. To see that the operator $P^B_t(\bA, h, \zeta)$ is symmetric and positive, we can suppose without loss of generality that $\bA, h, \zeta$ are smooth by the continuity in data proved in item 2. In this case, the symmetry and positivity follow from the observation that $s \mapsto \crochet{P^B_{s}(\bA, h, \zeta) g, P^B_{t-s}(\bA, h, \zeta) g}_{L^2}$ is differentiable for $t < t_0$ and
	\[\frac{\dd}{\dd s} \crochet{P^B_{s}(\bA, h, \zeta) g, P^B_{t-s}(\bA, h, \zeta) g}_{L^2} = 0.\]
	
	The above reasoning shows that all assertions in the statement hold up to a time horizon $t_0$. To complete the proof, we still have to extend the result to the whole time interval $[0, \infty)$. This can be achieved by an iterative process. Notice that the time horizon $t_0(\bA, q)$ depends only on the potential data $\bA, q$ but not on the initial condition $g$. One can therefore relaunch the equation with initial data $g' = P^B_{t_0} g$, yielding the solution to \eqref{eq:Landau_PDE} for $t \in [0, 2t_0]$. By iterating this reasoning, we obtain the global solution as well as the semigroup $(P^B_t)_{t \ge 0}$. Moreover, all assertions remain valid on $t \in [0, \infty)$ by construction.
\end{proof}

\begin{corollary}\label{cor:P_t-bounded}
	For $\bA \in \ccA$ and $q = (h, \zeta) \in \ccM$, $(P_t^B(\bA, q))_{t \ge 0}$ defines a strongly continuous semigroup satisfying the bound
	\begin{equ}\label{eq:P_t-bounded}
		\norm{P_t}_{\cL(L^2(B))} \leq M_h^B e^{\omega^B(\bA, q) t},
	\end{equ}
	for all $t \ge 0$, where $M_h^B$ is given by \eqref{eq:C-and-M} and $\omega^B(\bA, q)$ is some positive constant bounded by
	\[\omega^B(\bA, q) \lesssim (\log M_h^B) (1 + \norm{\bA}_{\ccA(B)} + \norm{q}_{\ccM(B)})^a,\]
	for some power $a > 0$.
\end{corollary}
\begin{proof}
	Take the constants $C^B(\bA, q)$ and $t_0$ given by \eqref{eq:C-and-M} and \eqref{eq:t0}. Recall from the bound \eqref{eq:M-1} that for all $t \leq t_0$, the fixed point $v$ for the contraction map $\cM$ satisfies the bound
	\[\normm{v}_{t} \lesssim \frac{\norm{g e^h}_{L^2}}{1 - C^B(\bA, q) t^{\frac{2-\theta-\kappa-\kappa'}{2}}}\;.\]
	With the bound on $\normm{v}_t$ we can estimate the $L^2$-norm of $v(t)$ through \eqref{eq:Landau_mild}: for $t \leq t_0$
	\begin{equs}
		\norm{v(t)}_{L^2(B)} &\leq \norm{g e^h}_{L^2(B)} + \int_0^t (t-s)^{-(\kappa+\kappa')/2} s^{-\theta/2} C^B(\bA, q) \normm{v}_{t} \dd s \\
		&\lesssim  \norm{g e^h}_{L^2(B)} + C^B(\bA, q) t^{\frac{2-\theta-\kappa-\kappa'}{2}} \normm{v}_{t}\\
		&\lesssim \frac{\norm{g e^h}}{1 - C^B(\bA, q) t^{\frac{2-\theta-\kappa-\kappa'}{2}}} \leq 2\norm{g e^h}_{L^2(B)}.
	\end{equs}
	This shows that, for $t \leq t_0$, the operator defined by $P_t^B(\bA, q) g = e^{-h} v(t)$ satisfies
	\[\norm{P_t^B}_{\cL(L^2(B))} \lesssim \norm{e^h}_{\cC^{1-\kappa}(B)} \norm{e^{-h}}_{\cC^{1-\kappa}(B)} =: C_h.\]
	Now consider $t \in ((n-1)t_0, nt_0]$ for some integer $n \ge 1$. We can estimate
	\begin{equs}
		\norm{P_t} \leq \norm{P_{t/n}}^n \lesssim C_h^n = C_h^{\ceil{t/t_0}} \leq (C_h \vee 1)^{t/t_0 + 1} = M_h^B e^{\omega^B(\bA, q) t},
	\end{equs}
	where $\omega^B(\bA, q) = t_0^{-1} \log M_h^B$. The bound for $\omega^B(\bA, q)$ then follows from the fact that $t_0^B(\bA, q)^{-1} = [2C^B(\bA, q)]^{2/(2-\theta-\kappa-\kappa')}$ and $C^B(\bA, q) \lesssim (1+ \norm{\bA}_{\ccA(B)} + \norm{q}_{\ccM(B)})^2$.
\end{proof}

Proposition \ref{prop:semigroup} yields the strongly continuous semigroup $(P_t^B)_{t > 0}$ satisfying the bound \eqref{eq:P_t-bounded}. By virtue of the Hille--Yosida theorem \cite[Theorem II.3.8]{EN00}, we give the following definition.
\begin{definition}
	Given an open box $B$, $\bA \in \ccA$ and $q \in \ccM$, the strongly continuous semigroup $(P_t^B(\bA, q))_{t > 0}$ admits a unique generator $\Phi^B(\bA, q)$ with a dense domain $\cD(\Phi^B(\bA, q)) \subset L^2(B)$, defined by
	\begin{equs}
		\Phi^B(\bA, q) f &= -\left.\frac{\dd}{\dd t} P_t^B(\bA, q) f\right|_{t = 0},\\
		\cD(\Phi^B(\bA, q)) &:= \{f\in L^2(B) \text{ for which the derivative exists}\}.
	\end{equs}
	
	In particular, $\Phi^B(\bA, q)$ is self-adjoint and its spectrum is contained in $(-\omega^B(\bA, q), \infty)$ with $\omega^B(\bA, q)$ given in Corollary \ref{cor:P_t-bounded}.
\end{definition}

\begin{corollary}\label{cor:norm-resolvent_conv}
	The map $\Phi^B: (\bA, q) \mapsto \Phi^B(\bA, q)$ is continuous in the sense that if $\bA_n \to \bA$ in $\ccA$ and $q_n \to q$ in $\ccM$, then $\Phi^B(\bA_n, q_n)$ converges to $\Phi^B(\bA, q)$ in the norm-resolvent sense. Furthermore, if $\bA \in \ccA$ and $h$, $\zeta$ are smooth functions, then
	\[\Phi^B(\bA, h, \zeta) = (i\nabla + \bA)^2 + |\nabla h|^2 - \Delta h + \zeta.\]
\end{corollary}
\begin{proof}
	The norm-resolvent convergence is a direct consequence of Proposition \ref{prop:semigroup} item 2. To prove the second claim, assume first that $\bA$, $h$ and $\zeta$ are smooth functions on $B$. It is well known by a classical argument that $S = (i\nabla + \bA)^2 + |\nabla h|^2 - \Delta h + \zeta$ defines a self-adjoint unbounded operator on $L^2(B)$. Moreover, it generates a semigroup $(e^h e^{-tS})_{t \ge 0}$ which solves the initial value problem \eqref{eq:Landau_PDE} for every $g \in L^2$. By uniqueness of Proposition \ref{prop:semigroup} and of the generator, we deduce that $\Phi^B(\bA, h, \zeta) = S$. To conclude, it remains to relax the assumption of smooth $\bA$ to $\bA \in \ccA$. This can be done by taking a sequence $\bA_n \in C^\infty(B)$ such that $\bA_n \to \bA$ in $\ccA$. Since in particular $\bA_n$ converges to $\bA$ uniformly in $B$, the quadratic form of $(i\nabla + \bA_n)^2 + |\nabla h|^2 - \Delta h + \zeta$ converges relatively boundedly to that of $(i\nabla + \bA)^2 + |\nabla h|^2 - \Delta h + \zeta$, from which the norm-resolvent convergence follows. The proof is then concluded by $\Phi^B(\bA, h, \zeta) = \lim_n \Phi^B(\bA_n, h, \zeta) = \lim_n (i\nabla + \bA_n)^2 + |\nabla h|^2 - \Delta h + \zeta = (i\nabla + \bA)^2 + |\nabla h|^2 - \Delta h + \zeta$ thanks to the continuity of $\Phi^B$ on $\bA$ in $\ccA$, where both limits are in the norm-resolvent sense.
\end{proof}

Since the operator $\Phi^B(\bA, q)$ has compact resolvent, it admits isolated eigenvalues with finite multiplicity, accumulating only at infinity, which we enumerate by
\[\lambda_1^B(\bA, q) \leq \lambda_2^B(\bA, q) \leq \lambda_3^B(\bA, q) \leq \cdots.\]

\begin{corollary}\label{cor:ev_continuity}
	For all $n \ge 1$, the eigenvalue $\lambda^B_n = \lambda^B_n(\bA, q)$ of $\Phi^B(\bA, q)$ is a locally Lipschitz function of $(\bA, q) \in \ccA \times \ccM$. Namely, for all $R>0$ and $\bA, q, \bar\bA, \bar q$ with norms bounded by $R$, one has
	\[|\lambda^B_n(\bA, q) - \lambda^B_n(\bar\bA, \bar q)| \leq C_R \left(\norm{q - \bar q}_{\ccM(B)} + \norm{\bA - \bar \bA}_{\ccA(B)}\right),\]
	for some $C_R > 0$.
\end{corollary}
\begin{proof}
	Denote by $\mu^B_1(\bA, q) \ge \mu^B_2(\bA, q) \ge \mu^B_3(\bA, q) \ge \cdots$ the eigenvalues of $P^B_t(\bA, q)$ in non-increasing order. Fix $t > 0$. Then by spectral calculus, one has the correspondence
	\[\lambda^B_n(\bA, q) = -\frac1{t} \log (\mu^B_n(\bA, q)).\]
	Since the map $x \mapsto -\frac1{t} \log(x)$ is locally Lipschitz from $\R_+$ to $\R$, it suffices to prove that $(\bA, q) \mapsto \mu^B_n(\bA, q)$ is locally Lipschitz from $\ccA \times \ccM$ to $\R_+$. This however follows from the fact that $P^B_t(\bA, q)$ is positive, Proposition \ref{prop:semigroup} item 2. (in particular, \eqref{eq:continuity_fixed-pt}) and the min-max formula
	\[\mu^B_n(\bA, q) = \sup_{V} \inf_{g \in V: \norm{g}_{L^2} = 1}\crochet{P^B_t(\bA, q) g, g},\]
	where the supremum runs over all vector subspaces $V$ of $L^2(B)$ with dimension $n$. Indeed, writing $f(\bA, q, g) := \crochet{P^B_t(\bA, q) g, g}$, one deduces from \eqref{eq:continuity_fixed-pt} that for every $R>0$, there exists a constant $C_R$ such that for all $\bA, q, g, \bar\bA, \bar q, \bar g$ whose norms are bounded by $R$,
	\[|f(\bA, q, g) - f(\bar\bA, \bar q, \bar g)| \leq C_R(\norm{g - \bar g}_{L^2} + \norm{q - \bar q}_{\ccM} + \norm{\bA - \bar \bA}_{\ccA}).\]
	It is then elementary to see that $(\bA, q) \mapsto \sup_V \inf_{g \in V: \norm{g} = 1} f(\bA, q, g)$ is locally Lipschitz.
\end{proof}

\begin{corollary}\label{cor:regularity_semigroup}
	Let $\bA \in \ccA$, $q = (h, \zeta) \in \ccM$ and $t > 0$. Then the operator $P^B_t(\bA, q)$ maps $L^2(B)$ continuously into the Hölder space $\cC^{\theta - 1}(B)$, as well as into the operator domain $\cD(\Phi^B(\bA, q))$ (and a fortiori into the form domain of $\Phi^B(\bA, q)$). In particular, the eigenfunctions of $\Phi^B(\bA, q)$ belong to $\cC^{\theta - 1}(B) \cap \cD(\Phi^B(\bA, q))$.
\end{corollary}
\begin{proof}
	For the first assertion, recall from the proof of Proposition \ref{prop:semigroup} that $P^B_t(\bA, q) g = e^{-h} v(t)$ with $v(t) \in \cH^{\theta}(B)$. In dimension two one has the embedding $\cH^{\theta}(B) \hookrightarrow \cC^{\theta - 1}(B)$, and multiplication by $e^{-h} \in \cC^{1-\kappa}(B)$ preserves $\cC^{\theta - 1}(B)$ since $0 < \theta - 1 < 1 - \kappa$; the resulting operator bound is of the same form as \eqref{eq:P_bound_sobolev}. For the second assertion, since $\Phi^B(\bA, q)$ is self-adjoint with spectrum contained in $(-\omega^B(\bA, q), \infty)$, the spectral calculus yields
	\[\norm{\Phi^B(\bA, q) P^B_t(\bA, q)}_{\cL(L^2(B))} \leq \sup_{\lambda > -\omega^B(\bA, q)} |\lambda| e^{-t\lambda} < \infty\]
	for every $t > 0$.
\end{proof}
\noindent Note that, when $q$ is taken to be the random model $Q(\xi)$ constructed in Section \ref{sec:Q} below, the domain $\cD(\Phi^B(\bA, Q))$ and the form domain are random subspaces of $L^2(B)$, into which the semigroup thus regularises instantaneously.

\subsection{Construction of $Q$}\label{sec:Q}
To prove Theorem \ref{thm1}, it remains to construct a sequence of random elements $Q^{(\eps)}(\xi) = (\rh^{(\eps)}(\xi), \rZ^{(\eps)}(\xi))$ of $\ccM$ with the properties that $\Phi(Q^{(\eps)}) = (i\nabla + \bA)^2 + \xi^{(\eps)} + C^{(\eps)}$ and that $Q^{(\eps)}$ converges to a non-trivial random variable $Q$ in $\ccM$ and in probability.\\

We proceed as in \cite{HL15}.
Let $G$ be a compactly supported, radially symmetric function that is smooth everywhere except at $0$ and such that
\begin{equ}
	-\Delta G = \delta + F
\end{equ}
for some $F \in C^\infty_c(\R^2)$.
(For example, one can set $G(x) = -\frac{\log|x|}{2\pi} \chi(x)$ for some suitable smooth cutoff function $\chi$.)

Let $\rho$ be a smooth radial function supported in the unit ball of $\R^2$ and define the mollifier sequence $\rho_\eps = \eps^{-2}\rho(\cdot/\eps)$.
Set $\xi^{(\eps)} = \xi \ast \rho_\eps$, and define
\begin{equ}\label{eq:model}
	\rh^{(\eps)} = G \ast \xi^{(\eps)}, \quad \rZ^{(\eps)} = C^{(\eps)} - |\nabla \rh^{(\eps)}|^2 - F\ast\xi^{(\eps)},
\end{equ}
where $C^{(\eps)}$ is chosen to be
\begin{equ}\label{eq:C}
	C^{(\eps)} = \E |\nabla \rh^{(\eps)} (0)|^2 = \int_{\R^2} |\nabla G\ast\rho_\eps(- y)|^2 \dd y.
\end{equ}

The following result is well known.
\begin{lemma}[Corollary 1.2 and Lemma 1.3 of \cite{HL15}]\label{lem:conv_noise}
	For the choice of $Q^{(\eps)}$ defined in \eqref{eq:model}, there exists a random variable $Q = (\rh, \rZ)$ taking values in $\ccM$ such that
	\[d_{\ccM}(Q^{(\eps)}, Q) \xrightarrow[\eps \to 0]{} 0 \quad \text{in probability.}\]
	The random function $\rh$ and the random distribution $\rZ$ are given by $\rh = G\ast\xi$ and $\rZ = -:|\nabla \rh|^2: - F\ast\xi$, where $:|\nabla \rh|^2:$ is the Wick square of the distribution $\nabla \rh$; more precisely, for any test function $\eta$, $\crochet{:|\nabla \rh|^2:, \eta}$ is represented by the second-order homogeneous Wiener chaos associated to the kernel
	\[(z, \bar{z}) \mapsto \int \nabla G(z - x) \cdot \nabla G(\bar{z} - x) \eta(x)\dd x. \]
\end{lemma}

\begin{proof}[Proof of Theorem \ref{thm1}.]
	With $Q^{(\eps)} = (\rh^{(\eps)}, \rZ^{(\eps)})$ as in \eqref{eq:model}, one has
	\[H_\eps = \Phi^B(\bA, Q^{(\eps)}) = (i\nabla + \bA)^2 + \xi^{(\eps)} + C^{(\eps)},\]
	for any $\bA \in \ccA$. By Lemma \ref{lem:conv_noise} and Corollary \ref{cor:norm-resolvent_conv}, it follows immediately that $H_\eps$ converges to the operator $\Phi^B(\bA, Q)$ in the norm-resolvent sense and in probability. The self-adjointness, the compactness of its resolvent and the fact that it satisfies the zero Dirichlet boundary condition follow easily from the construction of $\Phi^B$. The proof is thus completed.
\end{proof}

For later convenience, let us record an independence property of the eigenvalues $\lambda^B_n(\bA, Q)$ that follows from the construction of $Q = (\rh, \rZ)$.
\begin{lemma}\label{lem:independence}
	Fix $\bA \in \ccA$ and an integer $J \ge 1$. If $B_1,\dots,B_J$ are pairwise disjoint bounded open boxes, then the families $(\lambda^{B_j}_n(\mathbf{A},Q))_{n\ge1}$, $j=1,\dots,J$, are independent.
\end{lemma}
\begin{proof}
	Let us first remark that by \eqref{eq:model} and Corollary \ref{cor:norm-resolvent_conv}, one has $\Phi^{B_j}(\bA,Q^{(\eps)}) = (i\nabla+\bA)^2+\xi^{(\eps)}+C^{(\eps)}$ on $B_j$. Therefore, for $\eps > 0$ fixed, by the min-max formula for operators with smooth potential, each eigenvalue $\lambda_n^{B_j}(\bA,Q^{(\eps)})$ is Lipschitz continuous with respect to $\xi^{(\eps)}$ in the supremum norm over $B_j$, implying $\lambda_n^{B_j}(\bA,Q^{(\eps)})$ is a random variable measurable with respect to the $\sigma$-algebra $\cF_{\eps}^{B_j}$ defined as the completion of $\sigma\{\xi(\varphi): \varphi \in L^2 \text{ supported in $B^{(\eps)}_j$}\}$, where $B^{(\eps)}_j$ denotes the $\eps$-fattening of the open box $B_j$. We also write $\cF^{B_j}$ to denote the completion of $\sigma\{\xi(\varphi): \varphi \in L^2 \text{ supported in $B_j$}\}$.
	On the other hand, as $\lambda_n^{B_j}(\bA,Q^{(\eps)})$ converges to $\lambda_n^{B_j}(\bA,Q)$ in probability by Corollary \ref{cor:ev_continuity} and Lemma \ref{lem:conv_noise}, by extracting an a.s. convergent subsequence $\lambda_n^{B_j}(\bA,Q^{(\eps_m)})$ it follows that $\lambda_n^{B_j}(\bA,Q)$ is measurable with respect to $\bigcap_{m\ge 1} \cF_{\eps_m}^{B_j}$.
	We now argue that $\bigcap_{m\ge1}\cF^{B_j}_{\eps_m}=\cF^{B_j}$ up to $\P$-null sets.
	If true, then the assertion follows from the independence between $\cF^{B_j}$, $j = 1, 2, \dots, J$.
	
	Indeed, the desired equality of $\sigma$-algebras follows from the following property of white noise: for all $\varphi \in L^2(\R^2)$,
	\begin{equ}
		\E\left[e^{i\xi(\varphi)}\mid \cF^{B_j}_{\eps_m}\right] = \exp \left(i\xi(\varphi \un_{B^{(\eps_m)}_j}) - \frac12 \norm{\varphi \un_{(B^{(\eps_m)}_j)^c}}_{L^2}^2\right)\;,
	\end{equ}
	passing the left-hand side of which in $L^1(\P)$-limit yields $\E[e^{i\xi(\varphi)}\mid \bigcap_{m} \cF^{B_j}_{\eps_m}]$ since $\cF^{B_j}_{\eps_m}$ is a decreasing sequence of $\sigma$-algebras, while passing the right-hand side in $L^2(\P)$ yields $\exp (i\xi(\varphi \un_{B_j}) - \frac12 \|\varphi \un_{B_j^c}\|_{L^2}^2) = \E[e^{i\xi(\varphi)} \mid \cF^{B_j}]$. The uniqueness of the limit then implies $\E[e^{i\xi(\varphi)}\mid \bigcap_{m} \cF^{B_j}_{\eps_m}] = \E[e^{i\xi(\varphi)} \mid \cF^{B_j}]$ a.s., and thus $\bigcap_{m}\cF^{B_j}_{\eps_m} = \cF^{B_j}$ by the fact that the linear span of $\{e^{i\xi(\varphi)}\}$ is dense in $L^2(\P)$. This concludes the proof.
\end{proof}

\begin{remark}
	Note that for disjoint open boxes $B_1$ and $B_2$, the two restricted random functions $\rh\restriction_{B_1}$ and $\rh\restriction_{B_2}$ are in general \emph{correlated}! Indeed, the covariance of these two random functions is given by
	\begin{equ}
		\int_{\R^2} G(x - z) G(y - z) \dd z, \qquad x\in B_1,\; y \in B_2
	\end{equ}
	which is in general non-vanishing when the distance between $B_1$ and $B_2$ is smaller than the diameter of the support of $G$.
	Consequently, Lemma \ref{lem:independence} is non-trivial in the sense that it allows one to deduce the independence between $\lambda^{B_1}(\bA, Q)$ and $\lambda^{B_2}(\bA, Q)$ even though $Q\restriction_{B_1}$ and $Q\restriction_{B_2}$ are a priori correlated.
\end{remark}

\section{Asymptotics of $\lambda^L_n$ as $L \to \infty$} \label{sec:asymptotics}
We now turn to the proof of Theorem \ref{thm2}. Due to the independence of the random variable $\bA$ with respect to $\xi$ imposed by Assumption \ref{ass:A}, we will from now on, without loss of generality, fix $\bA$ to be a typical realisation in $\ccA$ satisfying the bound \eqref{eq:A_bound}.

Recall the notation for open box $B_{z, L} = z + (-L/2, L/2)^2$ centred at $z \in \R^2$. Recall also that $Q = (\rh, \rZ)$ is the $\ccM$-valued random variable constructed in Section \ref{sec:Q}. We claim that Theorem \ref{thm2} follows from the following probability tail estimate, whose proof we postpone to Section \ref{sec:tail_bound}.
\begin{theorem}\label{thm:ev_tail_estimate}
	Fix $\eta \in (0,1)$ and $z \in \R^2$. There exist constants $x_0>0$, $r > 2$ such that for all pairs $(L, x) \in \R_+^2$ verifying $x \ge x_0$ and $\frac{5r}{\sqrt{x}} \leq L \leq x^x$ the following inequalities hold
	\begin{equ}\label{eq:Tail}
		1-\exp \left(- \frac{x}{2r^2} e^{2\log L - (1 + \eta) \rho x}\right) \leq
		\P(-\lambda_{1}^{B_{z,L}}(\bA, Q) \geq x) \leq \frac{2x}{r^2} e^{2\log L - (1-\eta) \rho x}\;,
	\end{equ}
	where the constant $\rho$ is given by $\rho = 2C_{\mathrm{GN}}^{-1}$ with $C_{\mathrm{GN}}$ being the optimal constant of the Gagliardo--Nirenberg inequality in dimension $2$, as defined in \eqref{eq:Gagliardo--Nirenberg}.
\end{theorem}

\begin{remark}\label{rmk:tail_uniform}
	The constants $x_0$ and $r$ in Theorem \ref{thm:ev_tail_estimate} can be chosen uniformly over all centres $z$ with $|z| \leq c_0 L$, for any fixed $c_0 \ge 1$. Indeed, the centre enters the proof only through the bound \eqref{eq:X} on $A_{k, r, \beta}$ below, where the relevant points lie in $B_{z, L} \subset B_{0, 2(c_0 + 1)L}$, and Assumption \ref{ass:A} yields the same estimate with $\log(1+L)$ replaced by $\log(1 + 2(c_0+1)L)$, which only affects the implicit constant.
\end{remark}

\begin{proof}[Proof of Theorem \ref{thm2} assuming Theorem \ref{thm:ev_tail_estimate}]
	The proof goes along the same lines as \cite[Theorem 1]{HL23}. We first prove the asymptotics for $n = 1$.
	Write $a(L) = C_{\mathrm{GN}}\log(L)$. To prove the claim, it suffices to show that
	\[\P\left(\limsup_{L \to \infty} \left\{\left|\frac{\lambda^L_1(\bA, Q)}{a(L)} + 1\right| \ge \delta\right\}\right) = 0,\]
	for arbitrary $\delta \in (0, 1)$. For $m \ge 1$, set $L_m = 2^m$ and define the event
	\[E_m := \left\{\left|\frac{\lambda^{L_m}_1(\bA, Q)}{a(L_m)} + 1\right| \ge \delta\right\}.\]
	Note that the pair $(L_m, a(L_m))$ satisfies the assumption of Theorem \ref{thm:ev_tail_estimate} for sufficiently large $m$. A computation using the tail bounds then implies $\sum_{m \ge 1} \P(E_m) < \infty$. It follows from the Borel--Cantelli lemma that $\P(\limsup_{m \to \infty} E_m) = 0$. Finally, observe that
	\[ \frac{\lambda_{1}^{L_{m+1}}}{a(L_{m})} \leq \frac{\lambda_{1}^{L}}{a(L)} \leq \frac{\lambda_{1}^{L_m}}{a(L_{m+1})},\]
	for $L \in [2^m, 2^{m+1}]$ and on the event $\{\lambda^{L_m}_1 \leq 0\}$. The proof for the assertion is then completed by noting $a(L_m)/a(L_{m+1}) \to 1$ and that $\{\lambda_1^{L_m} \leq 0 \} \supset E_m^c$.
	
	For the case $n \ge 2$, fix $m$ such that $2^m \leq L < 2^{m+1}$ and let $D^m_1, \dots, D^m_n$ be $n$ disjoint open boxes of side length $2^m/n$ contained in $B_{0, 2^m}$, say translates of one another along the first coordinate axis. We claim that
	\[\lambda_1^{L}(\bA, Q) \leq \lambda_n^L (\bA, Q) \leq \max_{i = 1,\dots, n} \lambda_1^{D^m_i}(\bA, Q)\;. \]
	The first inequality holds since the eigenvalues are arranged in non-decreasing order. For the second, by the continuity of the eigenvalues (Corollary \ref{cor:ev_continuity}), it suffices to consider a smooth model $q = (h, \zeta)$, for which $\Phi^{B}(\bA, q) = (i\nabla + \bA)^2 + V$ with the smooth potential $V = |\nabla h|^2 - \Delta h + \zeta$ (Corollary \ref{cor:norm-resolvent_conv}), associated with the quadratic form $G^B(u) = \int_B |(i\nabla + \bA) u|^2 + V |u|^2$ with form domain $\cH^1_0(B; \C)$, the closure of $C^\infty_c(B; \C)$ in $\cH^1(B)$. Recall that the $n$-th eigenvalue (in non-decreasing order) of such an operator is given by the min-max formula:
	\begin{equ}
		\lambda_n^{L}(\bA, q) = \min_{W} \max_{u \in W: \norm{u}_{L^2} = 1} G^{B_{0, L}}(u)\;,
	\end{equ}
	where the $\min$ runs over all linear subspaces $W \subset \cH^1_0(B_{0, L}; \C)$ of dimension $n$. Since the boxes $D^m_i \subset B_{0, L}$ are disjoint, the ground states of the operators $\Phi^{D^m_i}(\bA, q)$, extended by zero to $B_{0, L}$, have disjoint supports and thus span a subspace $W_0 \subset \cH^1_0(B_{0, L}; \C)$ of dimension $n$. By locality of the form $G^{B_{0, L}}$ and the disjointness of the supports, every normalised $u \in W_0$ satisfies $G^{B_{0, L}}(u) \leq \max_{i = 1, \dots, n} \lambda_1^{D^m_i}(\bA, q)$, whence the second inequality.
	The centres of the boxes $D^m_i$ lie in $B_{0, 2^m}$, so by Theorem \ref{thm:ev_tail_estimate} and Remark \ref{rmk:tail_uniform}, each $\lambda_1^{D^m_i}(\bA, Q)$ satisfies the tail bound with side length $2^m/n$ and constants uniform in $m$. Running the Borel--Cantelli argument of the case $n = 1$ for each fixed $i$ along $m \to \infty$ then gives $\frac{\lambda_1^{D^m_i}(\bA, Q)}{a(2^m/n)} \to -1$ almost surely. Since the maximum of $n$ a.s.-convergent sequences converges a.s.\ and $\frac{a(2^m/n)}{a(L)} \to 1$ for $L \in [2^m, 2^{m+1})$, we conclude by squeezing.
\end{proof}

\subsection{Intermediate properties}\label{sec:ingredients}
We now collect several properties of the eigenvalues $\lambda_n^B$ necessary for the proof of Theorem \ref{thm:ev_tail_estimate}.

\subsubsection{Deterministic bounds on the eigenvalues}
We begin with some useful bounds on the eigenvalues of the operator. Specifically, Proposition \ref{prop:Landau-Anderson} compares the principal eigenvalue in the presence of a magnetic field to that without; Proposition \ref{prop:boxes} is an adaptation of a classical result which compares the principal eigenvalue on a large box $B_{0, L}$ to those on the sub-boxes.
\begin{proposition}\label{prop:Landau-Anderson}
	There exist constants $a, b, c >0$ such that for all $\bA \in \ccA$, $q \in \ccM$, all open boxes $B$ and all $x \in B$, one has the bound
	\[\lambda_1^{B}(0, q) \leq \lambda_1^{B}(\bA, q) \leq \lambda_1^{B}(0, q) + c A_{x, B} (M_h^{B})^2(|\lambda_1^{B}(0, q)|^a + C^{B}(0, q)^b) + A_{x, B}^2,\]
	where $A_{x, B} = \sup \{|\bA(y) - \bA(x)| : y \in B\}$ and $M_h^B$, $C^B(0, q)$ are given by \eqref{eq:C-and-M}.
\end{proposition}
\begin{proof}
	The lower bound is an easy consequence of the diamagnetic inequality \cite[Theorem 7.21]{LL01}.
	For the upper bound, we argue by first considering smooth models $q_\eps = (h_\eps, \zeta_\eps) \in C^\infty \times C^\infty$ converging to $q$ in $\ccM$ before passing to the limit by the continuity (Corollary \ref{cor:ev_continuity}). With $V_\eps = |\nabla h_\eps|^2 - \Delta h_\eps + \zeta_\eps$, notice that
	\begin{equs}
		\lambda_1^{B}(\bA, q_\eps) &= \inf_{u \in L^2(B): \norm{u}_{L^2} = 1} \int_{B} |(i\nabla + \bA)u|^2 + V_\eps |u|^2 \\
		&= \inf_{v \in L^2(B): \norm{v}_{L^2} = 1} \int_{B} |(i\nabla + \bA(\cdot) - \bA(x)) v|^2 + V_\eps |v|^2,
	\end{equs}
	where $x \in B$ is arbitrarily fixed and we have used in the second equality the transformation $u(y) = e^{i\bA(x) \cdot y} v(y)$ as well as the identity $(i\nabla + \bA) u = e^{i\bA(x) \cdot y} (i\nabla + \bA - \bA(x)) v$. With the notation $A_{x, B} = \sup \{|\bA(y) - \bA(x)| : y \in B\}$,
	\begin{equs}
		\int_{B} |(i\nabla + \bA(\cdot) - \bA(x))v|^2 \leq \int_{B} |\nabla v|^2 + 2A_{x, B} \norm{v}_{\cH^{1-\kappa}}^2 + A_{x, B}^2.
	\end{equs}
	Now take the $L^2$-normalised eigenfunction $\varphi_{q_\eps}^{B} \in \cH^{1-\kappa}$ associated with the ground state $\lambda^{B}_1(0, q_\eps)$, which is a minimiser of the variational problem $\inf \{\int_{B} |\nabla v|^2 + V_\eps|v|^2 : v \in L^2(B), \norm{v}_{L^2(B)} = 1\}$. We deduce
	\begin{equs}
		\lambda_1^{B}(\bA, q_\eps) &\leq \lambda^{B}_1(0, q_\eps) + 2A_{x, B} \norm{\varphi_{q_\eps}^{B}}_{\cH^{1-\kappa}}^2 + A_{x, B}^2.
	\end{equs}
	It remains to estimate the $\cH^{1-\kappa}$-norm of the eigenfunction $\varphi_{q_\eps}^B$. To this end, recall the constants $C^B(\bA, q)$, $M_h^B$, $t_0$ defined in \eqref{eq:C-and-M} and \eqref{eq:t0} with $\bA = 0$. For $t \leq t_0$, the bound \eqref{eq:P_bound_sobolev} implies
	\[\norm{e^{-t\lambda_1^B(0, q_\eps)}\varphi_{q_\eps}^B}_{\cH^{1-\kappa}} = \norm{P_{t}^B(0, q_\eps) \varphi_{q_\eps}^B}_{\cH^{1-\kappa}} \leq c M_{h_\eps}^B t^{-\theta/2},\]
	for some independent constant $c$ whose value may change from line to line, and we recall that the parameter $\theta$ is fixed in \eqref{eq:theta}. Thus
	\[\lambda_1^B(\bA, q_\eps) \leq \lambda^B_1(0, q_\eps) + cA_{x, B} (M_{h_\eps}^B)^2 t^{-\theta} e^{2 \lambda^B_1(0, q_\eps) t} + A_{x, B}^2.\]
	Let us choose $t = [|\lambda^B_1(0, q_\eps)| + t_0^{-1}]^{-1}$ so that $t \leq t_0$ and that $e^{\lambda^B_1(0, q_\eps) t} \lesssim 1$ uniformly over $\eps$. We then obtain the bound
	\[\lambda_1^B(\bA, q_\eps) \leq \lambda^B_1(0, q_\eps) + cA_{x, B} (M_{h_\eps}^B)^2 \left[|\lambda^B_1(0, q_\eps)| + (2C^B(0, q_\eps))^{\frac{2}{2 - \theta -\kappa - \kappa'}}\right]^{\theta} + A_{x, B}^2.\]
	Since $\lambda^B_1(0, q)$, $M_{h}^B$ and $C^B(0, q)$ are all continuous functions of $q$, a passage of $\eps$ to $0$ then concludes the proof.
\end{proof}

For $x \in \R^2$ and $r > 0$, recall the notation $B_{x, r}$ which denotes the open box $\{y \in \R^2 : |y - x|_\infty < r/2\}$.
\begin{proposition}\label{prop:boxes}
	Let $\bA \in \ccA$ and $q = (h, \zeta) \in \ccM$. Fix $L \ge 1$ and let $r < L$. Then,
	\[\min_{k \in N_2(L, r)} \lambda_1^{B_{kr, \frac{3r}{2}}}(\bA, q) - \frac{K}{r^2} \leq \lambda_1^{B_{0, L}}(\bA, q) \leq \min_{k \in N_1(L, r)} \lambda_1^{B_{kr, r}}(\bA, q),\]
	with a constant $K$ independent of $L, r, \bA, q$ and sets of indices $N_1(L, r), N_2(L, r) \subset \Z^2$ defined by
	\begin{equs}
		N_1(L, r) &= \left\{k \in \Z^2: B_{kr,r} \subset B_{0, L}\right\}\;,\\
		N_2(L, r) &= \left\{k \in \Z^2: B_{kr,\frac{3r}{2}} \cap B_{0, L} \neq \emptyset\right\}\;.
	\end{equs}
\end{proposition}
\begin{proof}
	It suffices to prove the bounds for smooth $q_\eps = (h_\eps, \zeta_\eps)$ converging to $q$ in $\ccM$, and then pass to the limit. Note that $\lambda_1(\bA, h_\eps, \zeta_\eps)$ is the lowest eigenvalue of $\Phi^B(\bA, h_\eps, \zeta_\eps) = (i\nabla + \bA)^2 + V_\eps$ with the smooth potential $V_\eps := |\nabla h_\eps|^2 -\Delta h_\eps + \zeta_\eps$.
	
	The upper bound is a simple consequence of the fact that $B_{0, L} \supset \bigcup_{k \in N_1(L, r)} B_{kr, r}$ and the variational formula
	\[\lambda_1^B(\bA, h_\eps, \zeta_\eps) = \inf_{\substack{u \in C^\infty_c(B;\C)\\\|u\|_{L^2} = 1}} G^B(u),\]
	where
	\[G^B(u) := \int_{B} |(i\nabla + \bA) u|^2 + V_\eps|u|^2 \]
	is the quadratic form associated to the operator.
	
	On the other hand, the lower bound can be derived from a slight modification of \cite[Proposition 1]{GK00}, which we describe now. We introduce a partition of unity $(\eta_k)_{k \in \Z^2}$ such that each $\eta_k$ is supported in $B_{kr, \frac{3r}{2}}$, $\sum_k \eta_k^2 = 1$ and that $\sum_k |\nabla \eta_k(x)|^2 \leq K/r^2$. Existence of such $(\eta_k)$ is elementary. For every $u \in C^\infty_c(B_{0, L}; \C)$ with unit $L^2$-norm, define $u_k := u \eta_k$ and we have $(i\nabla + \bA)u_k = \eta_k(i\nabla+ \bA)u + iu \nabla\eta_k$, leading to
	\[\sum_k |(i\nabla + \bA)u_k|^2 = |(i\nabla + \bA)u|^2 + \sum_k |\nabla \eta_k|^2|u|^2,\]
	where the cross term vanishes since $\sum_k \nabla (|\eta_k|^2) = \nabla (\sum_k \eta_k^2) = 0$. Therefore, one deduces 
	\[\sum_k G^{B_{0, L}}(u_k) = G^{B_{0, L}}(u) + \int_{B_{0, L}} \sum_k |\nabla\eta_k|^2|u|^2 \dd x \leq G^{B_{0, L}}(u) + \frac{K}{r^2},\]
	where the sum runs over $k \in \Z^2$ such that $B_{kr, \frac{3r}{2}} \cap B_{0, L} \neq \emptyset$. As the form $G^{B_{0, L}}$ is quadratic and $u_k/\|u_k\|$ is supported in $B_{kr, \frac{3r}{2}}$ and has unit $L^2$-norm, the left-hand side is bounded from below by 
	\[\sum_k \|u_k\|^2 G^{B_{0, L}}\left(\frac{u_k}{\|u_k\|}\right) \ge \min_k \lambda_1^{B_{kr, \frac{3r}{2}}}(\bA, h_\eps, \zeta_\eps).\]
	The proof is then concluded by taking an infimum over $\{u \in C^{\infty}_c(B_{0, L}; \C) : \|u\|_{L^2} = 1\}$.
\end{proof}

\subsubsection{Rescaling}
We now explore the probabilistic scaling properties of the eigenvalues. For $\beta > 0$, set $\bA_\beta(x) = \beta \bA(\beta x)$ and $V_\beta(x) := \beta^2 V(\beta x)$ for any function $V$. With $v(x) = u(\beta x)$, a simple calculation shows
\[(i\nabla + \bA_\beta)^2 v(x) + V_\beta(x) v(x) = \beta^2 [(i\nabla + \bA)^2 + V]u(\beta x).\]
Consequently, one expects that the equality $\lambda_n^{B/\beta}(\bA_\beta, V_\beta) = \beta^2 \lambda_n^B(\bA, V)$ holds, where $\lambda_n^B(\bA, V)$ denotes for the moment the $n$-th eigenvalue of $(i\nabla + \bA)^2 + V$ on $B$. Due to renormalisation, the same observation is not entirely true when the Gaussian white noise $\xi$ plays the role of $V$. We illustrate this point below.

Recall the notation $\lambda_n^B = \lambda_n^B(\bA, h, \zeta)$ as a continuous function of $\bA$, $h$, $\zeta$.

Notice that the Gaussian white noise on $\R^d$ satisfies
\[\xi(\beta \cdot) = \beta^{-d/2} \xi \quad \text{in law},\]
for all $\beta > 0$. One can build a pair $(\rh^{(\eps)}_\beta,  \rZ^{(\eps)}_\beta)$ corresponding to the rescaled white noise $\xi_\beta = \beta^2 \xi(\beta \cdot)$ as in \eqref{eq:model}. Let $\xi^{(\eps)}_\beta := \rho_\eps \ast \xi_\beta$. Note that $\xi^{(\eps)}_\beta$ is different from $\beta^2\xi^{(\eps)}(\beta\cdot)$. Instead, one has $\beta^2 \xi^{(\eps)}(\beta x) = \xi_\beta \ast \rho_{\eps/\beta}$, which has the same law as $\beta^{2-d/2} \xi^{(\eps/\beta)}(x)$. Define (with $d = 2$)
\begin{equs}
	\rh^{(\eps)}_\beta &:= G\ast \xi^{(\eps)}_\beta = \beta G\ast \xi^{(\eps/\beta)}\\
	\rZ^{(\eps)}_\beta &:= \tilde C^{(\eps)} - |\nabla \rh^{(\eps)}_\beta|^2 - F \ast \xi^{(\eps)}_\beta = \beta^{2} C^{(\eps/\beta)} - \beta^{2} |\nabla \rh^{(\eps/\beta)}|^2 - \beta F\ast\xi^{(\eps/\beta)},
\end{equs}
where the rightmost equalities in the above two equations are in the sense of law. By sending $\eps \to 0$, we obtain the limit in probability $Q_\beta = (\rh_\beta, \rZ_\beta)$. One has the following equalities in law
\begin{equs}
	\rh_\beta &= \beta G\ast \xi \label{eq:equality_law_h}\\
	\rZ_\beta &= - \beta^{2} :|\nabla \rh|^2: - \beta F \ast\xi. \label{eq:equality_law_zeta}
\end{equs}
\begin{remark}\label{rem:independence_Qbeta}
	Since $\xi_\beta$ is again a multiple of spatial Gaussian white noise on $\R^2$, Lemma \ref{lem:independence} also holds for $\lambda_n^{B_j}(\bA, Q_\beta)$ for any $\bA \in \ccA$ and any finite collection of disjoint open boxes $(B_j)_{j = 1, \dots, J}$.
\end{remark}
\begin{proposition}\label{prop:landau_scaling}
	Let $\bA \in \ccA$ and $Q = (\rh, \rZ) \in \ccM$ be defined as in Section \ref{sec:Q}.
	\begin{enumerate}
		\item If $\bA$ is in addition linear, i.e., $\bA(x + y) = \bA(x) + \bA(y)$, then the spectrum of $H$ is translation invariant in law. That is, for $y \in \R^2$ and a bounded open box $B \subset \R^2$, \[(\lambda_n^{y + B}(\bA, Q))_{n \ge 1} = (\lambda_n^B(\bA, Q))_{n \ge 1} \quad \text{in law}.\]	
		\item For $\beta > 0$ and $n \in \N$, one has
		\[ \lambda_n^{B/\beta}(\bA_\beta, Q_\beta) = \beta^2 \lambda_n^{B}(\bA, Q) + \delta_\beta  \quad \text{in law},\]
		where $\delta_\beta = \lim_{\eps \to 0} [\beta^{2} C^{(\eps/\beta)} - \beta^2 C^{(\eps)}] \asymp \beta^2 \log(\beta)$ as $\beta \to 0$, with $C^{(\eps)}$ given by \eqref{eq:C}.
	\end{enumerate}
\end{proposition}
\begin{proof}
	To show 1., we denote by $\tau_y$ the operator of translation by $y$, i.e., to every function $u$ one associates the function $\tau_y u$ given by $\tau_y u(x) = u(y + x)$. One can define $\tau_y \zeta$ by the usual extension to distributions. By linearity of $\bA$, we see that $\tau_y \bA(x) = \bA(x + y) = \bA(x) + \bA(y)$ for all $x, y \in \R^2$. Now observe that for any differentiable function $\phi: \R^2 \to \R$, we have $e^{i\phi} (i\nabla + \bA)^2 e^{-i\phi} = (i\nabla + \bA + \nabla \phi)^2$. In particular, by setting $\phi_y(x) = \bA(y) \cdot x$ for $x, y \in \R^2$, one has $(i\nabla + \tau_y \bA)^2 = e^{i\phi_y} (i\nabla + \bA)^2 e^{-i\phi_y}$, leading to
	\begin{equs}
		\Phi^{y + B}(\bA, h_\eps, \zeta_\eps) &= \tau_y^{-1} \Phi^B(\tau_y \bA, \tau_y h_\eps, \tau_y \zeta_\eps) \tau_y = \tau_y^{-1} e^{i\phi_y} \Phi^B(\bA, \tau_y h_\eps, \tau_y \zeta_\eps) e^{-i\phi_y} \tau_y.
	\end{equs}
	As the operator $u \mapsto e^{-i\phi_y} \tau_y u$ is unitary, we deduce that the spectra of the operators on both sides of the above equality coincide. As the law of white noise is translation invariant (hence so are those of $h_\eps$ and $\zeta_\eps$), we conclude that $(\lambda_n^{y + B}(\bA, h_\eps, \zeta_\eps))_n = (\lambda_n^{B}(\bA, h_\eps, \zeta_\eps))_n$ in law. The desired result then follows from the continuity of $\lambda$ with respect to $(h, \zeta) \in \ccM$.
	
	We now turn to point 2. By the fact that $\Phi^{B/\beta}(\bA_\beta, \rh^{(\eps)}_\beta, \rZ^{(\eps)}_\beta) = (i\nabla + \bA_\beta)^2 + \xi^{(\eps)}_\beta + \tilde C^{(\eps)}$, where $\tilde C^{(\eps)} = \beta^2 C^{(\eps/\beta)}$, and the computation in the beginning of this subsection, we deduce
	\[\lambda_n^{B/\beta}(\bA_\beta, \rh^{(\eps)}_\beta, \rZ^{(\eps)}_\beta) = \beta^2 \lambda_n^B(\bA, h_\eps, \zeta_\eps) + \tilde C^{(\eps)} - \beta^2 C^{(\eps)}.\]
	Since we know that $(\rh^{(\eps)}_\beta, \rZ^{(\eps)}_\beta)$ converges in $\ccM$ and in probability to the pair $(\rh_\beta, \rZ_\beta)$, the desired result then follows by the continuity of $\lambda$ on $\ccM$. The expression for $\delta_\beta$ is then given by $\delta_\beta = \lim_{\eps \to 0} (\tilde C^{(\eps)} - \beta^2 C^{(\eps)})$ and its asymptotics in $\beta$ follows from the fact that $\tilde C^{(\eps)} \asymp \beta^2 \log(\beta/\eps)$.
\end{proof}

\subsubsection{Large deviation principle}
With the pair $Q_\beta = (\rh_\beta, \rZ_\beta)$ defined in the previous subsection, we are interested in the large deviation behaviour of the corresponding eigenvalue $\lambda^B(0, Q_\beta)$ on a fixed box $B$, as $\beta \to 0$.

The following result is an application of the well-known large deviation principle (LDP) for Wiener chaos \cite{Led90} and the equalities in law \eqref{eq:equality_law_h}, \eqref{eq:equality_law_zeta}.
\begin{lemma}\label{lem:LDP_wiener-chaos}
	Fix a bounded open box $B$. As $\beta \to 0$, the law of $Q_\beta = (\rh_\beta, \rZ_\beta)$ induced on $\ccM(B)$ satisfies an LDP with rate $\beta^{2}$ and rate function $\cI$ given by
	\begin{equ}
		\cI(q) = \inf\left\{\frac12 \norm{\varphi}_{L^2(\R^2)}^2: \varphi \in L^2(\R^2), \; (G\ast\varphi, -|\nabla G\ast\varphi|^2 - F\ast\varphi)\restriction_{B} \; = q \right\}, \quad q \in \ccM(B),
	\end{equ}
	with the convention $\inf \emptyset = + \infty$. More precisely, it means that for any open set $O$ and any closed set $C$ in $\ccM(B)$, one has the inequalities
	\begin{equs}
		\liminf_{\beta \to 0} \beta^{2} \log \P(Q_\beta \in O) &\geq -\inf_{q \in O} \cI(q),\\
		\limsup_{\beta \to 0} \beta^{2} \log \P(Q_\beta \in C) &\leq -\inf_{q \in C} \cI(q).
	\end{equs}
\end{lemma}

\begin{remark}
	Note that the rate function $\cI$ is good \cite{DZ10} in the sense that the sublevel sets $\{\cI \leq c\}, c \in \R$ are compact. This follows from the fact that the map from $L^2(\R^2) \to \ccM(B)$ defined by $\varphi \mapsto (G\ast\varphi, -|\nabla G\ast\varphi|^2 - F\ast\varphi)\restriction_{B}$ is compact.
\end{remark}

\begin{proposition}[Large deviation]\label{prop:LDP}
	Fix a bounded open box $B$. As $\beta \to 0$, the sequence of random variables $\lambda_1^B(0, \rh_\beta, \rZ_\beta)$ satisfies an LDP with rate $\beta^{2}$ and the good rate function\footnote{Given any bounded open box $B$ and $\varphi \in L^2(\R^2)$, the operator $-\Delta + \varphi$ (with zero Dirichlet b.c. on $B$) is self-adjoint on $L^2(B)$ with compact resolvent. This can be seen by the construction given in Section \ref{sec:construction}. Indeed, it coincides with $\Phi^B(0,q)$ with $q = (G\ast\varphi, -F\ast\varphi - |\nabla G\ast\varphi|^2) \in \ccM$.}
	\begin{equ}\label{eq:rate_function}
		I_B(x) = \inf\left\{\frac12 \norm{\varphi}_{L^2(B)}^2: \text{$\varphi \in L^2(B)$ s.t. $x$ is the lowest e.v. of $-\Delta + \varphi$}\right\}, \quad x \in \R,
	\end{equ}
	with the convention $\inf \emptyset = \infty$.
\end{proposition}
\begin{proof}
	This is a consequence of Lemma \ref{lem:LDP_wiener-chaos}, the contraction principle \cite[Theorem 4.2.1]{DZ10}, the continuity of eigenvalues on $\ccM$ (Corollary \ref{cor:ev_continuity}) and the fact that 
	$\Phi^B(0,(G\ast\varphi, -|\nabla G\ast\varphi|^2 - F\ast\varphi)) = -\Delta + \varphi$
	for all $\varphi \in L^2(\R^2)$.
\end{proof}

\subsection{Proof of the tail bound Theorem \ref{thm:ev_tail_estimate}}\label{sec:tail_bound}
We now provide a proof of the tail bounds claimed in Theorem \ref{thm:ev_tail_estimate}.
Since the proof is identical for any $z \in \R^2$ at which the open box $B_{z, L}$ is centred, we will fix $z = 0$ in this proof and use the shorthand $\lambda^{L}_n := \lambda^{B_{0, L}}_n$.
Recall the $\ccM$-valued random variable $Q_\beta = (\rh_\beta, \rZ_\beta)$ defined in Section \ref{sec:ingredients}.

Let $r \in (2, L]$ be a sufficiently large constant which will be determined at the end of the proof. Denote $\beta = 1/\sqrt{x}$. Using the assumption $\frac{5r}{\sqrt{x}} \leq L$, we have $r \le L/\beta$. By Propositions \ref{prop:landau_scaling} and \ref{prop:boxes}, as well as the independence between eigenvalues on the boxes $B_{kr, r}$ (Lemma \ref{lem:independence} and Remark \ref{rem:independence_Qbeta}), one deduces
\begin{equs}
	\P(-\lambda_1^L(\bA, Q) \ge x) &= \P(\beta^2\lambda_1^L(\bA, Q) \le -1)\\
	&= \P(\lambda_1^{L/\beta}(\bA_\beta, Q_\beta) - \delta_\beta \le -1) \begin{cases}
		\leq \displaystyle \sum_{k \in N_2(L/\beta, r)} q_{kr, r}(\beta) \\ \\
		\geq \displaystyle 1 - \prod_{k \in N_1(L/\beta, r)} \left[1 - p_{kr, r}(\beta)\right]
	\end{cases}
\end{equs}
where we have used the shorthand
\begin{equs}
	p_{y, r}(\beta) &= \P(\lambda_1^{B_{y, r}}(\bA_\beta, Q_\beta) - \delta_\beta \le -1),\\
	q_{y, r}(\beta) &= \P(\lambda_1^{B_{y, \frac{3r}{2}}}(\bA_\beta, Q_\beta) - \frac{K}{r^2} - \delta_\beta \le -1),
\end{equs}
for $y \in \R^2$ and $r \in (2, L/\beta]$.
We now want to estimate the upper and lower bounds by Proposition \ref{prop:Landau-Anderson} and the LDP (Proposition \ref{prop:LDP}). By the lower bound of Proposition \ref{prop:Landau-Anderson}, we see that
\[q_{y, r}(\beta) \leq \P(\lambda_1^{B_{y, \frac{3r}{2}}}(0, Q_\beta) - \frac{K}{r^2} - \delta_\beta \le -1) = \P(\lambda_1^{\frac{3r}{2}}(0, Q_\beta) - \frac{K}{r^2} - \delta_\beta \le -1),\]
where we have used item 1 of Proposition \ref{prop:landau_scaling} (with $\bA = 0$) in the second equality. Notice that $\delta_\beta \to 0$ deterministically, implying that the laws of $\lambda_1^{\frac{3r}{2}}(0, Q_\beta) - \delta_\beta$ and $\lambda_1^{\frac{3r}{2}}(0, Q_\beta)$ are exponentially equivalent and hence they satisfy the same LDP (see \cite[Definition 4.2.10, Theorem 4.2.13]{DZ10}). In particular, $\lambda_1^{\frac{3r}{2}}(0, Q_\beta) - \delta_\beta$ satisfies an LDP with rate $\beta^2$ and a good rate function $I_{\frac{3r}{2}} := I_{B_{0,\frac{3r}{2}}}$ given by \eqref{eq:rate_function}.

Similarly, the upper bound of Proposition \ref{prop:Landau-Anderson} implies, for $k \in N_1$,
\begin{equs}
	p_{kr, r}(\beta) \ge \P(\lambda_1^{r}(0, Q_\beta) + X_{k,r,\beta} \le -1),
\end{equs}
where
\[X_{k,r,\beta} = cA_{k, r, \beta} (M_{\rh_\beta}^{B_{0, r}})^2 (|\lambda_1^{r}(0, Q_\beta)|^a + C^{B_{0,r}}(0, Q_\beta)^b) + A_{k, r, \beta}^2 - \delta_\beta,\]
with $A_{k, r, \beta} = \sup \{|\bA_\beta(x) - \bA_\beta(kr)| : x \in B_{kr, r}\}$. Here, we have used the invariance in law
\[(\lambda_1^{B_{kr, r}}(0, Q_\beta), M_{\rh_\beta}^{B_{kr, r}}, C^{B_{kr,r}}(0, Q_\beta)) = (\lambda_1^{r}(0, Q_\beta), M_{\rh_\beta}^{B_{0, r}}, C^{B_{0,r}}(0, Q_\beta)).\]
We argue that the laws of $\lambda_1^{r}(0, Q_\beta) + X_{k, r, \beta}$ and $\lambda_1^{r}(0, Q_\beta)$ are also exponentially equivalent uniformly for $k \in N_1$. This is the content of the following lemma.
\begin{lemma}\label{lem:LDP_exp-equiv}
	Let $r > 0$ be fixed. For all $\delta > 0$,
	\[\limsup_{\beta \to 0}\sup_{k \in N_1} \beta^{2}\log \P(|X_{k, r, \beta}| > \delta) = -\infty.\]
	In particular, $\lambda_1^{r}(0, Q_\beta) + X_{k, r, \beta}$ satisfies a uniform (for $k \in N_1$) LDP with rate $\beta^2$ and the good rate function $I_{r} := I_{B_{0, r}}$ defined by \eqref{eq:rate_function}.
\end{lemma}
\begin{proof}
	The key observation is the following uniform bound for $A_{k, r, \beta}$: notice that $B_{\beta kr, \beta r} \subset B_{0, L}$, hence under Assumption \ref{ass:A} and the condition $L \leq x^x$, or equivalently $\log L \leq \beta^{-2} \log(\beta^{-2})$, one has the bound
	\begin{equs}\label{eq:X}
		A_{k, r, \beta} \lesssim [\log (1 + L)]^\gamma \beta (\beta|x - kr|)^{\alpha} \lesssim r^{\alpha} \beta^{1+\alpha - 2\gamma} [\log(\beta^{-2})]^\gamma,
	\end{equs}
	uniformly for $k \in N_1(L/\beta, r)$ and for $\beta$ sufficiently small. Since we have imposed $1+\alpha - 2\gamma > 0$, the sequence $A_{k, r, \beta}$ converges to $0$ as $\beta \to 0$ uniformly for $k \in N_1(L/\beta, r)$.
	
	Notice that there exists some continuous function $f: \ccM(B_{0,r}) \to \R_+$ such that
	\[X_{k, r, \beta} = c A_{k, r, \beta} f(Q_\beta) + A_{k, r, \beta}^2 - \delta_\beta.\]
	Consequently, by Lemma \ref{lem:LDP_wiener-chaos} and the contraction principle, we deduce that $f(Q_\beta)$ satisfies an LDP with rate $\beta^2$ and the rate function
	\[\cI'(x) = \inf \left\{\frac12 \norm{\varphi}_{L^2(\R^2)}^2: \varphi \in L^2(\R^2), f((G\ast\varphi, -|\nabla G\ast\varphi|^2 - F\ast\varphi)\restriction_{B_{0,r}}) = x\right\}, \enspace x \in \R.\]
	Now fix arbitrarily $M > 0$. By \eqref{eq:X}, for $\beta$ sufficiently small, it holds
	\[\frac{\delta + \delta_\beta - A_{k, r, \beta}^2}{cA_{k, r, \beta}} > M,\]
	uniformly for $k \in N_1$. Hence, one has
	\begin{equ}[eq:LDP_exp-equiv_1]
		\limsup_{\beta \to 0} \sup_{k \in N_1} \beta^2 \log\P(X_{k, r, \beta} > \delta) \leq \limsup_{\beta \to 0} \beta^2 \log \P(f(Q_\beta) > M) \leq -\inf \cI'([M, \infty)).
	\end{equ}
	(The same bound for $\P(X_{k, r, \beta} < -\delta)$ can be obtained similarly.)
	
	Moreover, by Corollary \ref{cor:ev_continuity} and the definition of $M_{h}^{B_{0, r}}$ and $C^{B_{0,r}}(0, q)$, the function $f$ is locally bounded, namely has the property that
	\[\forall R> 0, \exists C_R>0, \quad \sup_{\norm{q}_{\ccM(B_{0, r})} \leq R} |f(q)| \leq C_R.\]
	It follows that $\sup_M \inf \cI'([M, \infty)) = + \infty$. (If not, set $R = \sup_M \inf \cI'([M, \infty)) < \infty$ and notice that for all $M > 0$, there exists $q \in \ccM$ such that $\norm{q}_{\ccM(B_{0, r})} \leq R$ and $|f(q)| \ge M$, which is absurd.) Since the bound \eqref{eq:LDP_exp-equiv_1} holds for all $M > 0$, the proof for the first assertion is then completed by taking the infimum of the right-hand side over $M$. For the ``in particular'' part of the statement, the proof is standard and follows from the fact that $I_{r}$ is a good rate function.
\end{proof}

We are now in a position to apply the LDPs. Set $\rho_r(a) := \inf \{I_r(x) : x \in (-\infty, a) \}$ and $\rho'_r(a) = \inf \{I_{r}(x) : x \in (-\infty, a] \}$ for $a \in \R$ and $r>0$. We now let $\eta \in (0, 1)$ be given and fix an arbitrary $\eps \in (0, \eta)$. By Proposition \ref{prop:LDP} and Lemma \ref{lem:LDP_exp-equiv}, there exists $\beta_0 > 0$ such that for all $\beta \leq \beta_0$
\begin{equs}
	p_{kr, r}(\beta) \ge e^{-\beta^{-2}\rho_{r}(-1)(1+\eps)}, \quad q_{kr, r}(\beta) \leq e^{-\beta^{-2} \rho_{\frac{3r}{2}}'(-1+\frac{K}{r^2}) (1-\eps)},
\end{equs}
where the inequality for $p_{kr, r}$ holds uniformly for $k \in N_1$ while that for $q_{kr, r}$ holds uniformly for $k \in N_2$.
Note that, provided $\frac{L}{\beta r} \ge 5$, the cardinalities of $N_1(L/\beta, r)$, $N_2(L/\beta, r)$ verify
\begin{equ}
	|N_1(L/\beta, r)| \ge \frac{1}{2} \left(\frac{L}{\beta r}\right)^2, \quad |N_2(L/\beta, r)| \le 2 \left(\frac{L}{\beta r}\right)^2\;.
\end{equ}
Now by $1 - y \leq e^{-y}$ for all $y \ge 0$, one has
\begin{equs}[eq:pq_bounds]
	\sum_{k \in N_2} q_{kr, r}(\beta) &\leq 2\left(\frac{L}{\beta r}\right)^2 e^{-\beta^{-2} \rho_{\frac{3r}{2}}'(-1+\frac{K}{r^2}) (1-\eps)},\\
	1 - \prod_{k \in N_1} \left[1 - p_{kr, r}(\beta)\right] &\ge 1- \exp\left(-\frac{1}{2}\left(\frac{L}{\beta r}\right)^2 e^{-\beta^{-2}\rho_r(-1)(1+\eps)}\right).
\end{equs}

Define $\rho = \inf_{r > 0} \rho_r(-1) = \lim_{r \to \infty} \rho_r(-1)$. In fact, the constant $\rho$ can be evaluated as $\rho = 2C_{\mathrm{GN}}^{-1}$, where $C_{\mathrm{GN}}$ is the constant given by \eqref{eq:Gagliardo--Nirenberg}. We refer the reader to \cite[Proposition 2.8]{HL23} or \cite[Theorem 2.6]{CvZ21} for a proof of this fact and other related properties. The only thing important to us is that $\rho$ is a positive constant and that $\rho_{\frac{3r}{2}}'(-1+\frac{K}{r^2}) \to \rho$ as $r\to \infty$.

As we have chosen $0 < \eps < \eta$, we can choose $r$ sufficiently large so that
\[- \rho'_{\frac{3r}{2}}(-1+\frac{K}{r^2}) (1-\eps) < -\rho(1-\eta), \quad - \rho_{r}(-1)(1+\eps) > -\rho(1+\eta).\]
Therefore, from \eqref{eq:pq_bounds} one deduces
\begin{equ}
	1- \exp\left(-\frac{1}{2}\left(\frac{L}{\beta r}\right)^2 e^{-\beta^{-2}\rho(1+\eta)}\right) \leq \P(-\lambda_1^L(\bA, Q) \ge x) \leq 2\left(\frac{L}{\beta r}\right)^2 e^{-\beta^{-2} \rho (1-\eta)},
\end{equ}
for all $\beta \leq \beta_0$. We conclude the proof by defining $x_0 = \beta_0^{-2}$ and substituting $\beta = x^{-\frac12}$.

\section{Generalisations and open questions}\label{sec:outlook}
We conclude with a heuristic discussion of several natural extensions of the present work, in increasing order of distance from our setting.

\paragraph{Dimension three.} In dimension $3$, the exponential transformation no longer removes all singular products, and the construction of the semigroup would require, for example, the theory of regularity structures. We expect, however, that the strategy of Section \ref{sec:construction} carries over: one lifts the mild formulation \eqref{eq:Landau_mild} to a fixed point problem in a space of Besov-based modelled distributions (for which the associated reconstruction theorem, the embeddings and the Schauder-type estimates are provided in \cite{HL17}) and feeds in the renormalised model for $-\Delta + \xi$ on boxes constructed in \cite{Lab19}. Upon reconstructing the solution in the Sobolev scale, the extraction of the semigroup and of its generator via the Hille--Yosida theorem is insensitive to the dimension. The treatment of the Dirichlet boundary condition in the parabolic setting might be the step requiring adaptation; cf. \cite{GH19a} for singular SPDEs in domains with boundaries. Carrying this out would in particular yield a full large deviation principle for the eigenvalues in dimension $3$, thereby removing the conditioning arguments of \cite{HL23}.

\paragraph{Neumann boundary conditions.} For $\bA = 0$, constructions with Neumann boundary conditions are available \cite{CvZ21, MvZ25}. In our approach, the Dirichlet condition plays a special role: it is preserved by the exponential transformation, since $u = e^{-h} v$ vanishes on $\partial B$ if and only if $v$ does. The Neumann condition on $u$, in contrast, would be twisted into the Robin condition $\partial_n v = (\partial_n h) v$ by the exponential transform $u = e^{-h}v$, whose coefficient involves the boundary trace of $\nabla h \in \cC^{-\kappa}$ which is ill-defined. A natural remedy is to adapt the kernel to the boundary, replacing $G$ by a localised Neumann Green function $G^{\mathrm{N}}$ built by even reflection across $\partial B$, so that the resulting field $\rh^{\mathrm{N}} = G^{\mathrm{N}} \ast \xi$ satisfies $\partial_n \rh^{\mathrm{N}} = 0$ in the sense of distributions. We therefore expect the construction to extend to the Neumann case. On the other hand, the arguments of Section \ref{sec:asymptotics} rely on the stationarity of the model, which is lost for the kernel $G^{\mathrm{N}}$, and their adaptation requires additional work.

\paragraph{White noise magnetic field.} One may wish to take the magnetic field $\nabla \times \bA$ to be itself a Gaussian white noise, independent of $\xi$, as considered in \cite{MM22}. In the Coulomb gauge, the associated vector potential $\bA$ is a Gaussian field of regularity $\cC^{-\delta}$ for every $\delta > 0$ and almost surely not a function. In this case $\nabla \cdot \bA = 0$, while the terms $|\bA|^2$ and $\nabla h \cdot \bA$ appearing in $F_2$ can no longer be defined as pathwise products. The former requires a logarithmically divergent Wick renormalisation as in \cite{MM22}, whereas the latter requires no counterterm: by independence, the would-be renormalisation constant vanishes, and $\nabla h \cdot \bA$ exists in the second-order Wiener chaos associated to two independent Gaussian white noises. In other words, these two terms would join $(h, \zeta)$ as components of an enhanced model. As for the asymptotics, Assumption \ref{ass:A} fails as stated, since pointwise evaluations of $\bA$ are no longer defined. A natural substitute is obtained by testing $\bA$ against test functions $\varphi^\lambda_z$ of scale $\lambda$ (a parameter replacing the difference $|x - y|$ in \eqref{eq:A_bound}) centred at points $z \in B_{0,L} \cap \Z^2$: a Gaussian computation shows that, almost surely, the supremum of $|\bA(\varphi^\lambda_z)|$ over all $z$ is of order $\sqrt{\log L} \lambda^{-\delta}$, which would formally correspond to the values $\alpha = -\delta$, $\gamma = \frac12$, just falling short of the condition $\gamma < \frac{1+\alpha}{2}$ imposed by Assumption \ref{ass:A}. This is in contrast with the Gaussian free field example, where the same $\sqrt{\log L}$ growth is compensated by the $(1-\kappa)$-Hölder regularity of $\bA$. Consequently, whether the conclusion of Theorem \ref{thm2} persists in this case is a delicate question that we leave open.

\paragraph{Full space.} Finally, one might wonder whether the operator could be constructed on the full space $\R^2$ by introducing suitable weights, as done for the Anderson Hamiltonian in \cite{HL15, HL18, HL26}.
On a bounded box, our approach treats the magnetic
terms as perturbations of the Laplacian. On $\R^2$ this is no longer
possible, since $\bA$ typically grows at infinity; already for a uniform
magnetic field, the spectrum consists of the Landau levels and differs
drastically from that of the Laplacian. A construction on the full space
would therefore require a solution theory built directly on the magnetic
Laplacian $(i\nabla+\bA)^2$. We intend to pursue this direction in the near
future.

\bibliographystyle{Martin}
\bibliography{ref}

\end{document}